\documentclass[a4paper, reqno]{amsart}
\usepackage[utf8]{inputenc}
\usepackage[T1]{fontenc}
\usepackage[english]{babel}
\usepackage{amsmath}
\usepackage{amsthm}
\usepackage{xfrac}
\usepackage{amssymb, enumerate}
\usepackage{amsfonts}
\usepackage{dsfont}
\usepackage{graphicx}
\usepackage[lofdepth,lotdepth]{subfig}
\usepackage{mathtools,mathrsfs} 
\usepackage{cases} 
\usepackage{tikz}
\usepackage[unicode,hypertexnames=false,colorlinks=true,linkcolor=blue,citecolor=blue]{hyperref}
\usepackage{xfrac}

\usepackage[shortlabels]{enumitem}

\usepackage{hyperref} 
\usepackage[capitalise, nameinlink, noabbrev]{cleveref} 
\hypersetup{colorlinks=true,
	linkcolor=blue,
	citecolor=red,
}

\newcommand{\eps}{\varepsilon}

\newtheorem{theorem}{Theorem}[section]
\newtheorem{proposition}[theorem]{Proposition}
\newtheorem{proposition and definition}[theorem]{Proposition and definition}
\newtheorem{definition}[theorem]{Definition}
\newtheorem{lemma}[theorem]{Lemma}

\theoremstyle{remark}
\newtheorem{remark}[theorem]{Remark}

\numberwithin{equation}{section}

\newcommand{\calM}{\mathcal{M}}

\newcommand{\R}{\mathbb{R}}
\newcommand{\N}{\mathbb{N}}

\newcommand{\rn}{\mathbb{R}^n}

\newcommand\abs[1]{\left|#1\right|}

\newcommand{\tend}[2]{\displaystyle\mathop{\longrightarrow}_{#1\rightarrow#2}}

\newcommand\1{\mathds{1}}
\newcommand\dom{d_{\partial \Omega}}

\DeclareMathOperator{\Exp}{Exp}

\title[Cut locus and variational inequalities]{Cut locus on compact manifolds and uniform semiconcavity estimates for a variational inequality}
\author{François Générau, \'Edouard Oudet, Bozhidar Velichkov}

\begin{document}
\maketitle
\tableofcontents

\begin{abstract}
 We study a family of gradient obstacle problems on a compact Riemannian manifold. We prove that the solutions of these free boundary problems are uniformly semiconcave and, as a consequence, we obtain some fine convergence results for the solutions and their free boundaries. Precisely, we show that the elastic and the $\lambda$-elastic sets of the solutions Hausdorff converge to the cut locus and the $\lambda$-cut locus of the manifold.

 %
 %

\end{abstract}

\section{Introduction}
Let $M$ be a smooth $n$-dimensional compact Riemannian manifold without boundary. Let $b \in M$ be a fixed point. We denote by $d_b:M\to\R$ the distance function to $b$,
and by $\text{\rm Cut}_b(M)$ the {\it cut locus}, that  is the set of points ({\it cut points})  $p\in M$ for which there exists a geodesic $\gamma$, starting from $b$ and passing through $p$, which is length minimizing between $b$ and $p$, but not after $p$.
The cut locus inherits much of the topology of $M$. It is a deformation retract of $M\setminus \{b\}$ and has the same homotopy type (see for instance \cite[Chapter III, Section 4]{sakai1996Riemannian}). Moreover, it is also related to the global geometry of $M$, for instance, to the geodesic spectrum (every close geodesics starting from $b$ crosses $\text{Cut}_b(M)$) and the Ambrose's problem (see \cite{hebdaMetricStructureCut1994}).

The local structure of the cut locus can be very rich and at the same time complicated, as it seems to be closely related to the regularity of $g$. A stratification theorem is available only when the metric $g$ is analytic (see \cite{myers1936connections} and \cite{buchner1977simplicial}), while in general, it is known that $\text{Cut}_b(M)$ must have an integer Hausdorff dimension (when $g$ is $C^\infty$) that  might even become fractional when $g$ is $C^k$ (see \cite{itohDimensionCutLocus1998} and the references therein).
The sensitivity with respect to the regularity of the manifold $(M,g)$ makes the cut locus difficult to recover by numerical methods involving discrete structures. A more stable object from this point of view is the so-called  {\it$\lambda$-cut locus} $\text{\rm Cut}_b^\lambda(M)$, which we introduce in this paper in analogy with the $\lambda$-medial axis of Chazal and Lieutier, which is a widely studied object in Computational Geometry (see \cref{sub:omega}). We refer to \cite{paper-num} for a detailed account on the impact of our study to the numerical methods for the computation of the cut locus.

For any $\lambda>0$, the $\lambda$-cut locus is defined as
\begin{equation}\label{eq:lambda cut locus}
\text{\rm Cut}_b^\lambda(M) := \left\{p\in M\setminus\{b\} : \abs{\nabla d_b(p)}^2 \leq 1 - \frac{\lambda^2}{d_b^2(p)}\right\},
\end{equation}
the norm of the generalized gradient $|\nabla d_b|$ being defined at every point $p\in M \setminus \{b\}$ as
\begin{equation}
\abs{\nabla d_b}(p) := \max\big\{0, \sup\limits_{v\in T_xM, \abs{v}=1} \partial^+_v d_b(p)\big\},
\end{equation}
where $\partial^+_v d_b(p)$ is the derivative of $d_b$ in the direction $v$ (see \cref{s:notation}). The $\lambda$-cut locus approximates the cut locus in the following sense: for every
$\lambda>0$, we have $\text{Cut}_b^\lambda(M)\subset \text{Cut}_b(M)$, while the closure of the union of $\text{Cut}_b^\lambda(M)$ over $\lambda>0$ is precisely $\text{Cut}_b(M)$ (see \cref{prop:approximation}). In particular, just as the cut locus, the $\lambda$-cut locus is a non-smooth set, with potentially very wild structure, even when $M$ is smooth.

\smallskip

In this paper we study the asymptotic behavior of a family of gradient obstacle problems on the manifold $M$ and we prove that both $\text{\rm Cut}_b(M)$ and $\text{\rm Cut}_b^\lambda(M)$ can be recovered from the solutions of these problems.
Moreover, even if our study is purely theoretical, it leads to a new method for the numerical approximation of the cut locus and the $\lambda$-cut locus on a compact manifold (see \cref{rem:num}).
\smallskip

For any $m>0$, we consider the variational minimization problem
\begin{equation}\label{eq:gradient constraint manifold}
\min\left\{\int_{M}  \abs{\nabla u}^2 - mu\ :\ u\in H^1(M),\ \abs{\nabla u} \leq 1,\ u(b)=0\right\}.
\end{equation}
This problem has a unique minimizer, which we will denote by $u_m$. We consider the sets
\begin{align}
E_m &:= \{p\in M\setminus \{b\} : \abs{\nabla u_m(p)} < 1\}, \nonumber
\\
\text{and} \quad E_{m,\lambda} &:= \left\{p\in M\setminus \{b\} : \abs{\nabla u_m(p)}^2 \leq 1 - \frac{\lambda^2}{u_m^2(p)}\right\}. \label{eq:lambda elastic set}
\end{align}
Our main result is the following.

\begin{theorem}[Approximation of $\text{\rm Cut}_b(M)$ and $\text{\rm Cut}_b^\lambda(M)$]\label{thm:main}
Let $M$ be a compact Riemannian manifold of dimension $n$ and let $b\in M$ and $\lambda>0$ be fixed. Then,
\begin{equation}\label{eq:convergence non-contact set}
	E_m\tend{m}{+\infty}\text{\rm Cut}_b(M) \quad \text{in the Hausdorff sense.}
	\end{equation}
Moreover, for any fixed $\eps>0$, we have that
	\begin{equation}
	\sup\limits_{p \in E_{m,\lambda}} d\big(p,\text{\rm Cut}_b^\lambda(M)\big) \tend{m}{+\infty}0, \label{eq:convergence lambda cut locus}
	\quad \text{and} \quad  \sup\limits_{p \in \text{\rm Cut}_b^{\lambda+\varepsilon}(M)} d(p,E_{m,\lambda}) \tend{m}{+\infty}0.
	\end{equation}
\end{theorem}

\begin{remark}[About the numerical computation of the cut locus]\label{rem:num}
We notice that the direct numerical approximation of the cut locus and the $\lambda$-cut locus is difficult and requires significant computational resources. Conversely, the variational problem \eqref{eq:gradient constraint manifold} consists in minimizing a convex functional under a convex constraint, which considerably simplifies this task. The numerical approach based on solving  \eqref{eq:gradient constraint manifold} will be the object of the forthcoming paper \cite{paper-num}.
\end{remark}

In order to prove \cref{thm:main}, we have to study the regularity of the solutions $u_m$ and the convergence of the asymptotic behavior (as $m\to\infty$) of the sequence $(u_m)$. We gather our results about the solutions of \eqref{eq:gradient constraint manifold} in the following theorem and we notice that \cref{thm:main} is in fact an immediate consequence of the claims \ref{item:uniform} and \ref{item:gradients} of \cref{t:main-main} below (see \cref{sub:proof-thm-main}).

\begin{theorem}[Regularity and convergence of $u_m$]\label{t:main-main}
Let $M$ be a compact Riemannian manifold of dimension $n$ and let $b\in M$ be fixed. Then, the following holds.
\begin{enumerate}[\rm(T1)]
\item\label{item:regularity} {\bf Regularity of $u_m$.} There exists a constant $m_0>0$, depending only on the manifold $M$, such that for every $m>m_0$, the minimizer $u_m$ of \eqref{eq:gradient constraint manifold} is locally $C^{1,1}$ on $M\setminus\{b\}$.
\item\label{item:Em} {\bf Properties of $E_m$.} For every $m\ge m_0$, $E_m$ is an open subset of $M$ and coincides with the set $\{u_m<d_b\}$. Moreover, $E_m$ contains $\text{\rm Cut}_b(M)$ and is at positive distance from $b$, that is $u_m=d_b$ in a neighborhood of $b$.
\item\label{item:monotonicity} {\bf Monotonicity of $u_m$ and $E_m$.} For every $m\ge m'\ge m_0$, we have $u_m\ge u_m'$. In particular, $E_m\subset E_{m'}$.
\item\label{item:semiconcavity} {\bf Semiconcavity of $u_m$.}
For every $\rho>0$, there are constants $C>0$ and $m_1>0$, depending on $\rho$ and on the manifold $M$, such that
\begin{equation}\label{eq:uniformly semiconcave manifold}
\text{$u_m$ is $C$-semiconcave on $M\setminus B_\rho(b)$,}
\end{equation}
for every $m\ge m_1$.

\item\label{item:uniform} {\bf Convergence of $u_m$.} The sequence $u_m$ converges uniformly on $M$ to the distance function $d_b$.

\item\label{item:gradients} {\bf Convergence of the gradients.} Let $p_\infty\in M\setminus\{b\}$. Then
\begin{itemize}
\item for every sequence $p_m\to p_\infty$, we have
\begin{equation}\label{e:inequality}
|\nabla d_b|(p_\infty)\le \liminf_{m\to\infty}|\nabla u_m|(p_m)\,;
\end{equation}
\item there exists a sequence $p_m\to p_\infty$ such that
\begin{equation}\label{e:equality}
|\nabla d_b|(p_\infty)= \lim_{m\to\infty}|\nabla u_m|(p_m)\,.
\end{equation}
\end{itemize}
\end{enumerate}

\end{theorem}

\begin{remark}
The semiconcavity of $u_m$ \ref{item:semiconcavity} and the convergence of the gradients \ref{item:gradients} are the most technical part of the proof and are precisely the properties that allow to approximate the $\lambda$-cut locus with the sets $E_{m,\lambda}$.
\end{remark}

\begin{remark}
If we replace the manifold $M$ with a smooth domain $\Omega\subset\R^n$ and $d_b$ with the distance to the boundary of $\Omega$, the problem \eqref{eq:gradient constraint manifold} becomes the classical elastic-plastic torsion problem, which we discuss in detail in \cref{sub:omega}. We notice that, for this problem, the claims \ref{item:regularity}, \ref{item:Em}, \ref{item:monotonicity} and \ref{item:uniform} are well-known. The elastic-plastic torsion problem has a long history and inspired the study of numerous problems involving more general (even fully nonlinear) operators. The crucial point in all these problems is that the gradient constraint in \eqref{eq:gradient constraint manifold} can be transformed into an obstacle constraint on the function (see \cref{sub:omega}). Until now, this property was exclusive for the Euclidean setting and for operators depending only on $\nabla u$ and $u$, but not on the points $x\in\Omega$ (in fact, for operators with variable coefficients, this equivalence is known to be false). A consequence of our analysis is that this crucial equivalence is not exclusively Euclidean but is a property of the underlying Riemannian structure of the manifold (see \cref{prop:equivalence constraints}).
%
%
\end{remark}

The rest of the introduction is organized as follows. In the next \cref{sub:omega} we will discuss the relation of the $\lambda$-cut locus and the problem \eqref{eq:gradient constraint manifold} to the $\lambda$-medial axis of Chazal-Lieutier and the classical elastic-plastic torsion problem. In \cref{sub:plan} we will discuss the key points in the proof of \cref{t:main-main} and the plan of the paper.

\subsection{Medial axis and $\lambda$-medial axis in a domain $\Omega$}\label{sub:omega}
This section is dedicated to the Euclidean counterpart of \cref{thm:main}. We go through the definitions of the medial axis and the $\lambda$-medial axis of a domain in the euclidean space. Then, we discuss the approximation theorem of Caffarelli and Friedman and its relation to \cref{thm:main}. Throughout this section, we will use the following notation: $\Omega$ is a bounded open set with $C^2$ regular boundary in $\R^n$ and $\dom :\Omega\to\R$ is the distance function to the boundary of $\Omega$,
\begin{equation*}
d_{\partial\Omega}(x):=\min\big\{|x-y|\ :\ y\in\partial\Omega\big\}.
\end{equation*}
\subsubsection{Definition of medial axis and $\lambda$-medial axis}
The {\it medial axis} $\mathcal M(\Omega)$ is defined as the set of points of $\Omega$ with at least two different projections on the boundary $\partial \Omega$,
$$\mathcal{M}(\Omega) := \big\{ x \in \Omega\ :\ \exists y,z \in \partial \Omega,\ \text{such that }\  y \neq z \; \text{ and } \; \dom(x) = \abs{x-y} = \abs{x-z} \big\}.$$

\noindent \begin{minipage}{0.72\textwidth}
	One crucial geometric property of the medial axis $\calM(\Omega)$ is that it is unstable with respect to small perturbations of the boundary of $\Omega$. For instance, the medial axis of the circle consists of its center only, while the medial axis of a polygonal approximation (the regularity of the approximating sets can be improved to $C^\infty$ by rounding the corners) is the star-shaped set on Figure 1. We refer to \cite{attali2009stability} for a detailed account on medial axis, stability and computability. This instability makes computing numerically $\calM(\Omega)$ quite tricky. Indeed, any numerical approximation of $\Omega$ (for instance, with polygons) might introduce an artificial (and large) medial set. In order to deal with this problem, in \cite{chazal_lieutier_lambda}, Chazal and Lieutier defined the so called {\it$\lambda$-medial axis} of $\Omega$ by setting, for any $\lambda>0$,
	\end{minipage}
\quad\begin{minipage}{0.24\textwidth}
\begin{center}
\vspace*{0.2cm}
			\begin{tikzpicture}
			\filldraw (0,0) circle [radius=1pt];
			\draw ({cos(0*180/8)},{sin(0*180/8)}) -- ({-cos(0*180/8)},{-sin(0*180/8)});
			\draw ({cos(180/8)},{sin(180/8)}) -- ({-cos(180/8)},{-sin(180/8)});
			\draw ({cos(2*180/8)},{sin(2*180/8)}) -- ({-cos(2*180/8)},{-sin(2*180/8)});
			\draw ({cos(3*180/8)},{sin(3*180/8)}) -- ({-cos(3*180/8)},{-sin(3*180/8)});
			\draw ({cos(4*180/8)},{sin(4*180/8)}) -- ({-cos(4*180/8)},{-sin(4*180/8)});
			\draw ({cos(5*180/8)},{sin(5*180/8)}) -- ({-cos(5*180/8)},{-sin(5*180/8)});
			\draw ({cos(6*180/8)},{sin(6*180/8)}) -- ({-cos(6*180/8)},{-sin(6*180/8)});
			\draw ({cos(7*180/8)},{sin(7*180/8)}) -- ({-cos(7*180/8)},{-sin(7*180/8)});
			\foreach \k in {0,...,15}
			\draw ({cos((\k )* 180/8)},{sin((\k )* 180/8)}) -- ({cos((\k +1)* 180/8)},{sin((\k +1) *180/8)});
			\end{tikzpicture}
\end{center}

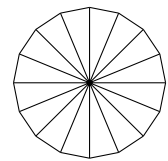
{\sc Figure 1. \it A polygonal\\ approximation of a circle,\\ with its medial axis.}\label{fig:medial axis}
\vspace*{0.2cm}
\end{minipage}

	\begin{equation}
\calM_\lambda (\Omega):= \{x\in\Omega : r(x)\geq\lambda\},
\end{equation}

\noindent where $r(x)$ is the radius of the smallest ball containing all the projections of $x$ on the boundary $\partial \Omega$, \textit{i.e.} the set $\{z\in \partial \Omega : \abs{x-z} = \dom(x) \}$. It is known that, for $\lambda$ small enough, $\calM_\lambda (\Omega)$ has the same homotopy type as $\calM (\Omega)$ (see \cite[section 3, theorem 2]{chazal_lieutier_lambda}) and that
\begin{equation*}
\calM(\Omega) = \bigcup_{\lambda>0}\calM_\lambda (\Omega).
\end{equation*}
These facts justify that $\calM_\lambda (\Omega)$ is a good approximation of $\calM(\Omega)$, for $\lambda$ small enough. The crucial  difference though is that $\calM_\lambda (\Omega)$ is stable with respect to small variations of $\Omega$, whereas $\calM(\Omega)$ is not (we refer to \cite[section 4]{chazal_lieutier_lambda} for precise statements and proofs).
Finally, we notice that the $\lambda$-medial axis can be equivalently defined (see \cite[section 2.1]{chazal_lieutier_lambda}) as
\begin{equation}\label{eq:equivalentdef}
\calM_\lambda (\Omega)= \left\{x\in \Omega : \abs{\nabla \dom(x)}^2 \leq 1 - \frac{\lambda^2}{\dom^2(x)}\right\},
\end{equation}
where $\nabla \dom$ denotes the generalized gradient wherever $\dom$ is not differentiable.

\bigskip

\subsubsection{Approximation of the medial axis}
Given a constant $m>0$ and a domain $\Omega$, as above, we consider the following {\it elastic-plastic torsion problem}
\begin{equation}\label{eq:elastoplastic}
\min\left\{\int_{\Omega}  \left(\abs{\nabla v}^2 - mv\right)dx\ :\ v\in H^1_0(\Omega),\ \abs{\nabla v} \leq 1\right\}.
\end{equation}
As in the case of \eqref{eq:gradient constraint manifold}, the problem \eqref{eq:elastoplastic} has a unique minimizer, which we will denote by $v_m$. Physically speaking, $v_m$ represents the stress function of a long bar of cross section $\Omega$, twisted with an angle $m$. The elastic-plastic torsion problem and the properties of its minimizer $v_m$ have been studied by various authors in the 60's and 70's (see for instance \cite{tingElasticplasticTorsionProblem1969a}, \cite{brezis_sibony_equivalence}, \cite{brezis1972multiplicateur}, \cite{caffarelliSmoothnessElasticPlasticFree1977}, \cite{tingElasticplasticTorsionProblem1977}, \cite{caffarelli_riviere_lipschitz}, \cite{friedman1980free} and \cite{caffarelliUnloadingElasticplasticTorsion1981}). In particular, in \cite{brezis_sibony_equivalence}, Brezis and Sibony proved that the gradient constraint in \eqref{eq:elastoplastic} can be replaced with an obstacle-type constraint on the function. Precisely, the minimizer $v_m$ of \eqref{eq:elastoplastic} is also the (unique) minimizer of
\begin{equation}\label{eq:elastoplasticequivalence}
\min\left\{\int_{\Omega}\left(\abs{\nabla v}^2 - mv\right)dx\ :\ v\in H^1_0(\Omega),\ v \leq \dom\right\}.
\end{equation}
Notice that this result was later generalized to a broader class of variational problems with convex constraints on the gradient (see \cite{treuEquivalenceTwoVariational2000}, \cite{maricondaGradientMaximumPrinciple2002} and \cite{safdariRegularityVectorvaluedVariational2018}). However, none of these will apply to our variant of the problem on manifolds, for which the equivalence of constraints fails in general (see Section \cref{s:esempio}).

\smallskip

Finally, using the equivalence of \eqref{eq:elastoplastic} and \eqref{eq:elastoplasticequivalence}, Caffarelli and Friedman (see \cite{caffarelli_friedman_elastoplastic}) proved that the sequence of {\it elastic sets}
$\{|\nabla v_m|<1\}\ $
Hausdorff converges, as $m\to+\infty$, to the medial axis $\mathcal M(\Omega)$.
To be precise, in \cite{caffarelli_friedman_elastoplastic}, it was showed that the elastic sets converge to the  so-called \emph{ridge} $\mathcal R(\Omega)$ which coincides with the closure of $\mathcal{M}(\Omega)$, when $\Omega$ has a $C^2$ regular boundary. This result from \cite{caffarelli_friedman_elastoplastic} is the euclidean counterpart of the first part of \cref{thm:main}. Nevertheless, the strategies from \cite{brezis_sibony_equivalence} and \cite{caffarelli_friedman_elastoplastic} cannot be reproduced on a manifold and do not imply the convergence of the $\lambda$-medial axis. In the proof of our \cref{thm:main}, we still aim at replacing the constraint on the gradient with a constraint on the function, but our approach is different and allows us to deal with the presence of the manifold and  to treat both the cut locus and the $\lambda$-cut locus.
In particular, we obtain the following approximation result for the $\lambda$-medial axis.

\begin{theorem}[Approximation of $\calM_\lambda (\Omega)$]\label{thm:omega}
	Let $\Omega$ be a bounded open set in $\R^n$ with $C^2$ regular boundary.  Then, setting
	\begin{equation*}
	E^\Omega_{m}=\big\{x\in \Omega : |\nabla v_m(x)|<1\big\}\qquad\text{and}\qquad E^\Omega_{m,\lambda}=\left\{x\in \Omega :  |\nabla v_m(x)|^2 \leq 1 - \frac{\lambda^2}{v_m^2(x)}\right\},
	\end{equation*}
	we have that, for any fixed $\eps>0$,
	\begin{equation*}
	\sup\limits_{x \in E^\Omega_{m,\lambda}} d\big(x,\calM_\lambda (\Omega)\big) \tend{m}{+\infty}0, 
	\quad \text{and} \quad  \sup\limits_{x \in \calM_{\lambda+\eps} (\Omega)} d(x,E^\Omega_{m,\lambda}) \tend{m}{+\infty}0.
	\end{equation*}
\end{theorem}

\subsection{Proof of Theorem \ref{t:main-main} and plan of the paper}\label{sub:plan}

 We consider the variational problem
\begin{equation}\label{eq:distance constraint}
\min\left\{\int_{M}  \abs{\nabla u}^2 - mu\ :\ u\in H^1(M),\ u \leq d_b\right\}.
\end{equation}
It is immediate to check that \eqref{eq:distance constraint} admits a minimizer and that this minimizer is unique (this follows by the convexity of the functional and the constraint). We will denote by  $u_m^d:M\to\R$ ('d' stands for the 'distance' constraint) the unique minimizer of \eqref{eq:distance constraint}.

 \subsubsection{ Part I. Equivalence of \eqref{eq:gradient constraint manifold} and \eqref{eq:distance constraint}} Our first aim is to show that the problems \eqref{eq:gradient constraint manifold} and \eqref{eq:distance constraint} are equivalent, that is the minimizers $u_m$ and $u_m^d$ are the same. Now, since every function which is $1$-Lipschitz and is zero in $b$ stands below the distance function $b$, it is clear that $u_m$ can be used to test the optimality of $u_m^d$, that is, we have
$$\int_M\big(|\nabla u_m^d|^2-mu_m^d\big)\le \int_M\big(|\nabla u_m|^2-mu_m\big).$$
Notice that, if we are able to prove that the minimizer $u_m^d$ is $1$-Lipschitz, then we can use $u_m^d$ to test the minimality of $u_m$, i.e.
$$\int_M\big(|\nabla u_m^d|^2-mu_m^d\big)\ge \int_M\big(|\nabla u_m|^2-mu_m\big).$$
This gives that both $u_m$ and $u_m^d$ are solutions of \eqref{eq:gradient constraint manifold} (and also of \eqref{eq:distance constraint}), which means that they have to coincide. Thus, in order to prove that \eqref{eq:gradient constraint manifold} and \eqref{eq:distance constraint} are equivalent, we have to prove that
\begin{equation}\label{e:grad-u-m-d-bound}
|\nabla u_m^d|\le 1\quad\text{on}\quad M.
\end{equation}
In order to prove this, we proceed as follows:
\begin{itemize}
\item First, we prove that $u_m^d$ is $C^1$-regular locally in $M\setminus \{b\}$ (see \cref{prop:regularity}).
\item Then, from \cref{lemma:u=d around b} and \cref{lemma:regularized distance}, we deduce that
$$\text{\rm Cut}_b(M)\subset \{u_m^d<d_b\}\subset M\setminus\{b\}.$$
In particular, since $d_b$ is smooth away from $\{b\}$ and $\text{\rm Cut}_b(M)$, we get that on the boundary $\partial \{u_m^d<d_b\}$ both the distance function $d_b$ and the solution $u_m^d$ are differentiable and have the same gradient, which entails that $|\nabla u_m^d|=1$ on $\partial \{u_m^d<d_b\}$.
\item Finally, we use the fact that $u_m^d$ solves the PDE
$$\Delta u_m^d=m\quad\text{in}\quad \{u_m^d<d_b\},\qquad |\nabla u_m^d|=1\quad\text{on}\quad \{u_m^d=d_b\}$$
to deduce that $|\nabla u_m^d|\le 1$ also in the set $\{u_m^d<d_b\}$. Now, in the flat (Euclidean) case, this inequality is an immediate consequence of the fact that $|\nabla u_m^d|^2$ is subharmonic. On a general manifold $M$ the situation is more complicated as the curvature comes into play in the computation of  $\Delta\big(|\nabla u_m^d|^2\big)$. For this reason we are able to prove the bound $|\nabla u_m^d|\le 1$ on $M$ (and so the equivalence of the two problems) only in the case when $m$ is large enough. Before we give the precise statement of this result (see \cref{prop:equivalence constraints}), let us emphasize that this is not a mere technical assumption, but a consequence of the geometry of the manifold. In fact, in the appendix (\cref{prop:counterexample}), we give an example of a $2$-manifold $M$ for which the bound on the gradient fails when $m$ is small.
\end{itemize}

The following is the key result for the analysis of the solution of problem \eqref{eq:gradient constraint manifold}. The proof is given in \cref{sec:equivalence constraints}.

\begin{proposition}[Equivalence of \eqref{eq:gradient constraint manifold} and \eqref{eq:distance constraint}]\label{prop:equivalence constraints}
Let $M$ be an $n$-dimensional compact Riemannian manifold and let the constant $K\geq 0$ be a lower bound for the Ricci curvature :
	\begin{equation}
	\mathrm{Ric}\geq -K,
	\end{equation}
	where $\mathrm{Ric}$ denotes the Ricci curvature tensor of $M$. Then, for every
	\begin{equation}\label{e:bound-from-below-m-equivalence}
	m\geq \frac{1}{2}\max\Big\{\sqrt{nK(1+K\mathrm{diam}(M)^2)},\ n K \mathrm{diam}(M)\Big\},
	\end{equation}
	we have that
	\begin{equation}\label{e:inequality-grad-u-m-d}
	\abs{\nabla u_m^d}=1\quad \text{on}\quad \{d_b=u_m^d\},\qquad\text{and}\qquad  \abs{\nabla u_m^d}<1\quad \text{in}\quad E_m^d:=\{u_m^d<d_b\}.
	\end{equation}
	In particular, for $m$ as in \eqref{e:bound-from-below-m-equivalence}, we have that $u_m^d=u_m$, where $u_m$ is the minimizer of \eqref{eq:gradient constraint manifold}.
\end{proposition}

Finally, as a corollary of \cref{prop:equivalence constraints}, we obtain the first two claims of \cref{t:main-main}.

\begin{proof}[Proof of \cref{t:main-main} \ref{item:regularity} and \ref{item:Em}]
 By \cref{prop:equivalence constraints} we have that $u_m=u_m^d$. From the regularity of $u_m^d$ (\cref{prop:regularity}, \cref{lemma:u=d around b} and \cref{lemma:regularized distance}), we obtain \ref{item:regularity} and \ref{item:Em}.
\end{proof}

Moreover, as in the classical case of the elastic-plastic torsion problem (see \cite{caffarelli_friedman_elastoplastic}), we can now use the structure of \eqref{eq:distance constraint} to obtain information about the monotonicity of $E_m$ and the uniform convergence of $u_m$.

\begin{proof}[Proof of \cref{t:main-main} \ref{item:monotonicity} and \ref{item:uniform}]
The uniform convergence $u_m^d\to d_b$ on $M$, as $m\to \infty$, is proved in \cref{lemma:uniform convergence}. The monotonicity of $u_m$ and $E_m$, and the Hausdorff convergence of $E_m$ to $\text{\rm Cut}_b(M)$, now follow from \cref{thm:non-contact set converges to cut locus}.
\end{proof}


\subsubsection{Part II. Uniform semiconcavity and convergence of the gradients} We recall that our final objective is to prove the convergence of the sets $E_{m,\lambda}$ (\cref{thm:main}) and $E_{m,\lambda}^\Omega$ (\cref{thm:omega}). Now, from the definition of $E_{m,\lambda}$, it is clear that this boils down to proving a convergence result for the gradients $|\nabla u_m|$. On the other hand, we cannot expect any uniform estimate on the modulus of continuity of $|\nabla u_m|$; in fact, the sequence $u_m$ converges (uniformly) to the distance function $d_b$, which is not even differentiable at all points. Thus, we adopt a different strategy and we prove that the solutions are uniformly semiconcave, where our definition of semiconcavity is the following.
\begin{definition}[$C$-semiconcavity]\label{d:semiconcavity}
	Given a constant $C>0$, a function $u$ is said to be $C$-semiconcave on $M$ if and only if for any unit speed geodesic $\gamma:[a,b] \to M$, the function $t \mapsto C t^2 - u(\gamma(t))$ is convex.
	Moreover,
	\begin{itemize}
	\item we say that $u$ is semiconcave if it is $C$-semiconcave for some constant $C>0$;
	\item we say that $u$ is locally semiconcave if for any $p \in M$, $u$ is semiconcave in a neighborhood of $p$.
	\end{itemize}
\end{definition}

The most technical result of the paper is \cref{t:main-main} \ref{item:semiconcavity}, which we prove in \cref{sec:semiconcavity}. The key result is \cref{prop:semiconcavity} and applies to both \cref{t:main-main} and \cref{thm:omega}. Let us briefly give the idea of the proof of this proposition here, directly in the setting of \cref{t:main-main} \ref{item:semiconcavity}.\smallskip

\begin{proof}[Sketch of the proof of \cref{t:main-main} \ref{item:semiconcavity}]  First, we fix a constant $C_d$ such that the distance function $d_b$ is $C_d$-semiconcave on $M\setminus B_\rho(b)$. Then, for every unit speed geodesic $\gamma:[a,b]\to M$, and every $\lambda\in[0,1]$, we define the function
\begin{align*}
c(\gamma,\lambda)
&:= \lambda(1-\lambda)(C_d+1)(b-a)^2  -\Big((1-\lambda)u_m(\gamma(a)) + \lambda u_m(\gamma(b))-u_m(\gamma(\lambda_{ab}))\Big),
\end{align*}
where $\lambda_{ab} = (1-\lambda)a +\lambda b$.
We will show that the minimum of this function over all geodesics $\gamma$ and all $\lambda$ is positive, which will give that $u$ is $(C_d+1)$-semiconcave. First, we show that for any unit speed geodesic $\gamma$ and $\lambda\in(0,1)$, we can build a unit speed geodesic $\widehat{\gamma}:[a,b]\to M$ and $\widehat{\lambda}\in (0,1)$, such that
\begin{equation*}
	c(\widehat{\gamma},\widehat{\lambda},u_m)\leq c(\gamma,\lambda,u_m) \quad \text{and} \quad \widehat{\gamma}(a,b) \subset E_m = \{u_m< d_b\}.
\end{equation*}
This follows from the semiconcavity of $d_b$ and the inequality $u_m\le d_b$ (this is explained in detail in the proof of \cref{prop:semiconcavity}). Thus, we only need to show the semiconcavity of $u_m$ in the non-contact region $E_m$. Since $u_m$ is smooth in $E_m$, we need to prove that (see \cref{prop:D2})
$$D^2u_m\leq (C_d+1)Id\quad\text{in}\quad E_m.$$
In order to prove this inequality, for every $p\in E_m$ and $X\in \mathbb S^{n-1}(T_pM)$ we consider an auxiliary function of the form
\begin{equation*}
f_\varepsilon(p,X) := D^2u_m(X,X) + \varepsilon\left(C_1 \abs{\nabla u_m}^2(p) + C_2 u_m^2(p) - C_3 u_m(p)\right),
\end{equation*}
and we show that for $\varepsilon>0$ small enough and $m$ large enough, we have $\displaystyle f_{\varepsilon}\leq C_d+ \sfrac{1}{2}.$ We suppose that the maximum of $f_\eps$ is achieved for some $q\in E_m$ and some $Y\in \mathbb S^{n-1}(T_qM)$ (the case when the minimum is achieved for $q\in \partial E_m$ is a consequence of known estimates for the solutions of the obstacle problem with variable coefficients, see \cref{s:gradients}). Then, we construct, locally around $q$, a function of the form
$$p\mapsto f_\eps(p,X(p))\quad\text{where}\quad X(p)\in  \mathbb S^{n-1}(T_pM),$$
and we compute its Laplacian in the variable $p$ (notice that in the flat euclidean case we can simply take the section $p\mapsto X(p)$ to be constant). Finally, we obtain that for an appropriate choice of $\eps$ and $m$, the Laplacian of this function has to be positive, which contradicts the minimality of $q$ and concludes the proof.
\end{proof}

The main part of the proof of \cref{t:main-main} \ref{item:semiconcavity} is contained in \cref{prop:semiconcavity}, which applies to both \cref{t:main-main} and \cref{thm:omega}. In the proof of \cref{prop:semiconcavity}, the  function $c$ is the Riemannian counterpart of the Korevaar's convexity function (see \cite{korevaar_maxprinc});
in computing the Laplacian of $f_\eps(p,X(p))$ we use some of Guan's second order estimates for Hessian equations in Riemannian manifolds (see \cite{guanSecondOrderEstimates2014}).

\medskip

At this point, the convergence of the gradients $|\nabla u_m|$ (\cref{t:main-main} \ref{item:gradients}) follows from the uniform semiconcavity of $u_m$ by a general argument (we give the proof of this fact in \cref{s:gradients}). We are now in position to prove \cref{thm:main}.
%
%

%

\subsection{Proof of \cref{thm:main}}\label{sub:proof-thm-main} The Hausdorff convergence of the elastic sets $E_m$ to $\text{\rm Cut}_b(M)$ is a consequence from the uniform convergence (\cref{t:main-main} \ref{item:uniform}) of the solutions $u_m$ to the distance function $d_b$, as explained in \cref{thm:non-contact set converges to cut locus}.
Let us now prove the first claim in \eqref{eq:convergence lambda cut locus}.  Suppose by contradiction that there are a constant $\delta>0$, a sequence $m_k\to\infty$ and a sequence of points $p_k$ such that
\begin{equation}\label{e:contra1}
p_k\in E_{m_k,\lambda}\quad\text{and}\quad d\big(p_k,\text{\rm Cut}_b^\lambda(M)\big) >\delta.
\end{equation}
By the facts that $M$ is compact and that $u_{m_k}$ coincides with the distance function $d_b$ in a neighborhood of $b$ (that does not depend on $k$), we may suppose that $p_k$ converges to some $p_\infty\in M\setminus\{b\}$.  Now, from the uniform convergence of $u_{m_k}$ and \cref{t:main-main} \ref{item:gradients}, we get that
$$|\nabla d_b|(p_\infty)\le\liminf_{k\to\infty}|\nabla u_{m_k}|(p_k)\le \lim_{k\to\infty}\left(1-\frac{\lambda^2}{u_{m_k}^2(p_k)}\right)=1-\frac{\lambda^2}{d_b^2(p_\infty)},$$
which means that $p_\infty\in \text{\rm Cut}_b^\lambda(M)$, in contradiction with \eqref{e:contra1}.

Suppose now that the second claim in \eqref{eq:convergence lambda cut locus} does not hold.
Then, there are a constant $\delta>0$, a sequence $m_k\to\infty$ and a sequence of points $p_k\in M\in\{b\}$ such that
\begin{equation*}
p_k\in \text{\rm Cut}_b^{\lambda+\eps}(M)\qquad\text{and}\qquad d\big(p_k,E_{m_k,\lambda}\big) >\delta\quad\text{for every}\quad k\ge 0.
\end{equation*}
Up to extracting a subsequence, we may suppose that $p_k$ converges to a point $p_\infty$ such that
\begin{equation}\label{e:contra2}
p_\infty\in  \text{\rm Cut}_b^{\lambda+\eps}(M)\qquad\text{and}\qquad d\big(p_\infty,E_{m_k,\lambda}\big) >\frac\delta2\quad\text{for every}\quad k\ge 0.
\end{equation}
Now, by \cref{t:main-main} \ref{item:gradients}, there is a sequence $q_k\to p_\infty$ such that
$$|\nabla d_b|(p_\infty)=\lim_{k\to\infty}|\nabla u_{m_k}|(q_k).$$
In particular, since $p_\infty\in \text{\rm Cut}_b^{\lambda+\eps}(M)$, we have
$$\lim_{k\to\infty}\left(|\nabla u_{m_k}|(q_k)-1+\frac{\lambda^2}{u_{m_k}^2(q_k)}\right)=|\nabla d_{b}|(p_\infty)-1+\frac{\lambda^2}{d_{b}^2(p_\infty)}\le -\frac{2\eps\lambda+\eps^2}{d_{b}^2(p_\infty)}.$$
Thus, the left-hand side is negative for $k$ large enough and so, we have $q_k\in E_{m_k,\lambda}$, which is a contradiction with \eqref{e:contra2}. This concludes the proof of \cref{thm:main}. \qed

\subsection{Proof of \cref{thm:omega}}\label{sub:proof-thm-omega} As shown in \cref{sec:semiconcavity}, we may apply \cref{prop:semiconcavity} to get that the functions $v_m$ are uniformly semiconcave on $\Omega$. It is already known that the solution $v_m$ of \eqref{eq:elastoplastic} and \eqref{eq:elastoplasticequivalence} is locally $C^{1,1}$ on $\Omega$.
It is also well-known that $v_m$ converges uniformly to $d_{\partial\Omega}$ as $m\to\infty$. As a consequence, reasoning as in \cref{s:gradients}, we get that for every $x_\infty\in\Omega$, the following holds:
\begin{itemize}
	\item if $x_m\to x_\infty$, then $\displaystyle\ |\nabla d_{\partial\Omega}|(x_\infty)\le \liminf_{m\to\infty}|\nabla v_m|(x_m)\,;$
	\item there exists a sequence $x_m\to x_\infty$ such that
$\displaystyle\ |\nabla d_{\partial\Omega}|(x_\infty)= \lim_{m\to\infty}|\nabla v_m|(x_m)\,.$
\end{itemize}

Now, the conclusion follows as in the proof of \cref{thm:main}.\qed

\subsection*{Acknowledgements}
The three authors were partially supported by Agence Nationale de la Recherche (ANR) with the projects GeoSpec (LabEx PERSYVAL-Lab, ANR-11-LABX-0025-01), CoMeDiC (ANR-15-CE40-0006) and ShapO (ANR-18-CE40-0013). The third author is supported by the European Research Council (ERC), under the European Union’s Horizon 2020 research and innovation programme, through the project ERC VAREG - \it Variational approach to the regularity of the free boundaries \rm (grant agreement No. 853404).

\section{Notation, definitions and preliminary results}\label{s:notation}

\subsection{General notation}
We will denote by $g$ the metric on $M$.  $TM$ denotes the tangent bundle of $M$ and $T_pM$ the tangent space of $M$ at $p$. By $\mathbb S^{n-1}(T_pM)$ we will denote the unit sphere in $T_pM$, that is
$$\mathbb S^{n-1}(T_pM):=\big\{X\in T_pM\ :\ g(X,X)=1\big\}.$$
$\Exp:TM\to M$ is the global exponential map, while $\exp_p$ is its restriction to $T_pM$. Finally, given a function $u$ on $M$, $Du$ is the differential of $u$, $\nabla u$ is the gradient, and $D^k u$ is the $k$-th covariant derivative (in particular, by $D$ we denote also the Riemannian connection on $M$). Thus, for smooth vector fields $X,Y:M\to TM$, we have
$$g(\nabla u,X):=Du(X)=D_Xu=Xu\qquad\text{and}\qquad D^2u(X,Y)=g(D_X(\nabla u),Y).$$
We will also use the notation $|\nabla u|^2$ for $g(\nabla u,\nabla u)$, and $\Delta u$ for the Laplace-Beltrami operator on $M$. We notice that $-\Delta$ is positive, that is, we have the integration by parts formula
$$\int_M g(\nabla u,\nabla v)=\int_M (-\Delta u) v,$$
for every $u,v\in C^2(M)$. Unless otherwise specified, all the integrals will be taken with respect to the volume form associated to the Riemannian metric $g$. Finally, we recall that $H^1(M)$ denotes the usual space of Sobolev functions on $M$, which is the closure of $C^1(M)$ with respect to the $H^1$-norm defined as
$$\|u\|_{H^1}^2=\int_M |\nabla u|^2+\int_M u^2.$$

\subsection{Semiconcave functions}
In this section, we gather some of the main properties of semiconcave functions on smooth Riemannian manifolds, which we will need in the proof of \cref{t:main-main}.
Some of these results can be found in \cite{petruninSemiconcaveFunctionsAlexandrov2006}, in the context of Alexandrov spaces, while for a more detailed introduction to semiconcave functions in the framework of euclidean spaces we refer to \cite{cannarsa2004semiconcave}.
\medskip

Let $M$ be a Riemannian manifold, $u:M\to\R$ a given function and $\gamma:[a,b] \to M$ be a curve in $M$. It is immediate to check that the function
$$t\mapsto Ct^2-u(\gamma(t))$$
is convex on $[a,b]$ if and only if
\begin{equation}\label{e:semiconcavity}
(1-\lambda)u(\gamma(a)) + \lambda u(\gamma(b)) - u(\gamma(\lambda_{ab})) \leq C \lambda (1-\lambda) (b-a)^2\quad\text{for any}\quad \lambda \in [0,1],
\end{equation}
where here and throughout the paper, we use the notation
\begin{equation}\label{e:lambda_a_b}
\lambda_{ab}:= (1-\lambda)a +\lambda b\qquad\text{for any}\qquad a,b,\lambda\in \R.
\end{equation}
In particular, this means that the function $u$ is $C$-semiconcave on $M$ if and only if \eqref{e:semiconcavity} holds for any unit speed geodesic $\gamma:[a,b] \to M$. Analogously, $u$ is locally semiconcave if for every $p\in M$ there is a geodesic ball $B_\rho(p)$ and a constant $C_p>0$ such that \eqref{e:semiconcavity} holds (with $C=C_p$) for every unit speed geodesic $\gamma:[a,b]\to B_\rho(p)$.
\begin{remark}\label{rem:local_semiconcavity}
On a compact Riemannian manifold, semiconcavity and local semiconcavity are the same.
\end{remark}

\begin{proposition}[Semiconcavity in terms of $D^2u$]\label{prop:D2}
Let $u:M\to\R$ be $C^2$-regular. Then
	$$D^2u\leq 2 C\,\text{ on }\, M\qquad\text{if and only if}\qquad\text{$u$ is $C$-semiconcave on $M$}.$$
\end{proposition}
\begin{proof}
Let $\gamma:[a,b]\to M$ be a unit speed geodesic. Then the function $t\mapsto Ct^2-u(\gamma(t))$ is convex if and only if
$$0 \leq 2C-\frac{d^2}{dt^2}u(\gamma(t))= 2C-\frac{d}{dt} Du(\dot\gamma(t))= 2C-\Big(D^2u(\dot\gamma(t),\dot\gamma(t))+Du(D_{\dot\gamma(t)}\dot\gamma(t))\Big)= 2C-D^2u\big(\dot\gamma(t),\dot\gamma(t)\big).$$
The claim follows.
\end{proof}

The semiconcavity can also be read in local coordinates as follows.
\begin{proposition}[Semiconcavity in local coordinates]\label{prop:definitions semiconcavity_intro}
Let $u:M\to \R$ be a locally Lipschitz function on a Riemannian manifold $M$. Then, $u$ is locally semiconcave if and only if for any chart $\psi$ of $M$, $u \circ \psi ^{-1}$ is locally semiconcave as a function on $\rn$.
\end{proposition}
We postpone the proof of this proposition to \cref{app:semiconcavity}. We next show that we can define the gradient of a semiconcave function at every point.


\begin{proposition}[The generalized gradient of a semiconcave function]\label{prop:generalized gradient}
	Let $u:M \to \R$ be a locally Lipschitz and semiconcave function. Then, at every point $p\in M$, $u$ admits a directional derivative $\partial^+_v u(p)$ in any direction $v\in T_pM\setminus\{0\}$. It is defined by
	\begin{equation*}
	\partial^+_v u(p) := \frac{\mathrm{d}}{\mathrm{d}t} [u(\gamma(t))]_{t=0}=\lim_{t\to0^+}\frac{u(\gamma(t))-u(p)}{t},
	\end{equation*}
	where $\gamma : [0,1] \to M$ is any curve such that $\gamma(0)=p$ and $\dot{\gamma}(0) = v$\,.
	Moreover, the map $v \mapsto \partial^+_v u(p)$ is $1$-homogeneous and concave on $T_pM$. Thus, it attains a unique maximum in the closed unit ball of $T_pM$ at a unique vector $v_p$.
\end{proposition}
\begin{proof}
By Proposition \ref{prop:definitions semiconcavity_intro}, we can suppose that $M=\R^n$, $p=0$ and that $\gamma(t)=tv$. Then the function $w(x)=C|x|^2-u(x)$ is convex for $C$ large enough and so, the function
$t\mapsto \frac{w(\gamma(t))-w(0)}{t}$ is non-decreasing in $t$, so the limit $\partial^+_v u(p)=-\partial^+_v w(0)$ exists and is finite. The convexity of the function $v\mapsto \partial^+_v w(0)$ is a consequence from the convexity of $w$. The existence of a maximum of $v \mapsto \partial^+_v u(p)$ follows.
\end{proof}
If $\partial^+_{v_p} u(p)>0$, then the $1$-homogeneity implies that $v_p$ has norm one, and we define
	\begin{equation*}
\nabla u(p) := \partial^+_{v_p} u(p) v_p\qquad\text{and}\qquad |\nabla u(p)|=\partial^+_{v_p} u(p) .
\end{equation*}
If $\partial^+_{v_p} u(p)=0$, then we set $\nabla u(p)=0$. Thus, the norm of $\nabla u(p)$ is given by the following formula:
	\begin{equation}\label{e:def-gradient}
	\abs{\nabla u(p)} = \max\big\{0,\max\limits_{v\in T_pM, \abs{v}=1} \partial^+_{v} u(p)\big\}.
	\end{equation}

\subsection{Distance function, cut locus and cut points} Let $M$ be a compact Riemannian manifold, $b\in M$ and $d_b:M\to\R$ be the distance function to $b$.
Here we recall the definition and some of the main properties of the cut locus.

\begin{definition}[Cut points]
	Let $T>0$ and $\gamma : [0,T] \to M$ be a unit speed geodesic such that $\gamma(0) = b$, $t_0\in (0,T)$ and $p = \gamma(t_0)$. We say that $p$ is a \emph{cut point} of $b$ along $\gamma$ if $\gamma$ is length minimizing between $b$ and $p$, but not after $p$, \textit{i.e} $d_b(\gamma(t)) = t$ for $t\leq t_0$, and $d_b(\gamma(t)) < t$ for $t>t_0$.
\end{definition}
\begin{definition}[Cut locus]
	The \emph{\it cut locus} of $b$ in $M$, $\text{\rm Cut}_b(M)$, is defined as the set of all cut points of $b$.
\end{definition}
\noindent The following well-known facts about the cut locus can all be found in \cite[Chapter III, Section 4]{sakai1996Riemannian}:
\begin{itemize}
\item $\text{\rm Cut}_b(M)$ is the closure of the set of points $p$ in $M$, for which there are at least two minimizing geodesics connecting $b$ and  $p$;
\item the distance function $d_b$ is smooth outside $\text{\rm Cut}_b(M)\cup\{b\}$ and
$$|\nabla d_b|=1\quad\text{in}\quad M\setminus\Big(\text{\rm Cut}_b(M)\cup\{b\}\Big)\ ;$$
\item $d_b$ is differentiable at $p\in M$ if and only if there is a unique minimizing geodesics between $b$ and $p$;
\item in particular, $\text{\rm Cut}_b(M)\cup \{b\}$ is the closure of the set of points of non differentiability of $d_b$;
\item the exponential map $\exp_b:T_bM\to M$ is a diffeomorphism from an open set of $T_bM$ onto $M \setminus \text{\rm Cut}_b(M)$;
\item $\text{\rm Cut}_b(M)$ is a deformation retract of $M\setminus \{b\}$. In particular, these two sets have the same homotopy type, and so $\text{\rm Cut}_b(M)$ inherits much of the topology of $M$ (like homology groups, for instance). See \cite[Chapter III, Section 4, Proposition 4.5]{sakai1996Riemannian} for a precise statement.
\end{itemize}


We next recall that in \cite[Proposition 3.4]{mantegazza_mennucci_2003}, it was proved that, for any chart $\psi$  on $M\setminus\{b\}$, the function $d_b \circ \psi ^{-1}$ is locally semiconcave on $\rn$. Thus, by \cref{prop:definitions semiconcavity_intro}, $d_b$ is locally semiconcave on $M\setminus\{b\}$ in the sense of \cref{d:semiconcavity}. Precisely, we have the following proposition
\begin{proposition}[Semiconcavity of the distance function]\label{prop:db semiconcave}
	Let $M$ be a compact Riemannian manifold of dimension $n$ and $b\in M$ be a given point. Then, for every $\rho>0$, there is a constant $C>0$ such that the distance function $d_b$ is $C$-semiconcave on $M\setminus B_\rho(b)$.
\end{proposition}
In particular, by \cref{prop:generalized gradient}, for any point $p\in M\setminus \{b\}$ and any direction $v\in T_pM$, $d_b$ admits the directional derivative $\partial^+_{v} d_b(p)$ and so we can define $\nabla d_b$ and $|\nabla d_b|$ at every point as in \eqref{e:def-gradient}. In \cref{prop:gradient less than 1} we give a geometric interpretation of $\abs{\nabla d_b}(p)$ in terms of the geodesics connecting $p$ to $b$. We notice that similar results holds also in the more general framework of Alexandrov spaces, but with some additional restrictions on the curvature of the ambient space (see  \cite[Theorem 4.5.6]{adelsteinMorseTheoryUniform2017} and also \cite[Lemma 3.2]{adelsteinMorseTheoryUniform2017} for the statement  in the Riemannian context). We give the proof directly for the distance function to a compact subset $K$ of $M$.

\begin{lemma}[Geometric interpretation of the generalized gradient]\label{prop:gradient less than 1}
	Let $M$ be a smooth Riemannian manifold without boundary, $K$ a compact subset of $M$, and $d_K$ the distance function to $K$. Let $p$ be a point of $M$ such that there exist several minimizing geodesics from $p$ to $K$. We denote the set of unit speed geodesics from $p$ to $K$ that are minimizing between $p$ and $K$ by $\mathrm{geod}(p,K)$.
	For any $v \in T_pM$, we have
	\begin{equation}\label{e:geometric-main}
	\partial_v^+ d_K(p) = \min_{\gamma\in geod(p,K)} - \dot{\gamma}(0) \cdot v.
	\end{equation}
	In particular,
	\begin{equation}\label{e:geometric-cor}
	\abs{\nabla d_K}(p) = \max\{0,\max_{v\in T_pM, \abs{v} =1 }\min_{\gamma\in geod(p,K)} - \dot{\gamma}(0) \cdot v\}.
	\end{equation}
In particular, if $\gamma_1:[0,d_K(p)]\to M$ and $\gamma_2:[0,d_K(p)]\to M$	are two minimizing geodesics from $p$ to $K$, then
\begin{equation}\label{e:geometric-two-geodesics}
|\nabla d_K(p)|\le \sqrt{\frac{1+\dot\gamma_1(0)\cdot \dot\gamma_2(0)}2}.
\end{equation}
\end{lemma}
\begin{proof}
	Let $\gamma : [0,d_K(p)]\to M$ be a geodesic of $geod(p,K)$. Let $a = \gamma(d_K(p)/2)$. As $\gamma$ is minimizing between $p$ and $\gamma(d_K(p))$, we have $a\notin Cut_p(M)$, and so $p\notin Cut_a(M)$. In particular, the function $d_a$ is differentiable at $p$, and $\nabla d_a(p) = - \dot{\gamma}(0)$. Thus, for every $t>0$, we have
	\begin{align*}
	\frac{d_K(\exp_p(tv)) - d_K(p)}{t}
	& \leq \frac{d_a(\exp_p(tv)) +d_K(a) - d_K(p)}{t}  = \frac{d_a(\exp_p(tv)) - d_a(p)}{t} 
	\end{align*}
Passing to the limit as $t\to0$, we get
	\begin{equation}\label{eq:upper bound directional derivative}
	\partial_v^+ d_K(p) \leq \min_{\gamma\in geod(p,K)} - \dot{\gamma}(0) \cdot v.
	\end{equation}
	Now, for every $t>0$, let $\gamma_t \in geod(\exp_p(tv),K)$. For $t$ small enough, the length of $\gamma_t$ is bounded by $d_K(p) + 1$. By compactness of the set of geodesics of length bounded by a given constant, there exists a sequence of positive numbers $(t_n)_{n\geq 0}$ that converges to $0$, such that $\gamma_n := \gamma_{t_n}$ converges to a unit speed geodesic $\gamma$ as $n\to +\infty$. As $K$ is closed, $\gamma$ is a geodesic from $p$ to $K$. What is more, we have
	\begin{equation*}
	\mathrm{length}(\gamma) = \lim_{n\to\infty}\mathrm{length}(\gamma_n) = \lim_{n\to\infty}d_K(\exp_p(t_n v)) = d_K(p),
	\end{equation*}
	so $\gamma\in geod(p,K)$. Let $R = \min\{\mathrm{inj}(M),d_K(p)/2\}$, where $\mathrm{inj}(M)$ is the injectivity radius of $M$. In particular for any $(x,y)$ such that $d(x,y)<R$ and $x \neq y$, the distance function $d(\,\cdot\,,\,\cdot\,)$ is smooth in a neighborhood of $(x,y)$ in $M \times M$. For $n \in \N$, let $b_n := \gamma_n(R)$, and $b_\infty = \gamma(R)$.
	Let ${U},{V}\subset M$ be precompact neighborhoods of $p$ and $b_\infty$ respectively such that $d(\,\cdot\,,\,\cdot\,)$ is smooth on $\overline{U}\times \overline{V}$.
	For $n$ big enough, we have $\exp_p(t_nv) \in U$ and $b_n \in V$, and so
	\begin{align}
	d_K(p)
	& \leq d_K(b_n) + d(b_n,p) \nonumber \\
	& = d_K(\exp_p(t_n v)) - d(b_n,\exp_p(t_n v)) + d(b_n,p) \nonumber \\
	& = d_K(\exp_p(t_n v)) - \nabla_2 d(b_n,p)\cdot v + o(t_n), \label{eq:nabla2}
	\end{align}
	where $\nabla_2$ is the gradient with respect to the second variable.
	We have
	\begin{equation*}
	\nabla_2 d(b_n,p)\tend{n}{\infty}\nabla_2 d(b_\infty,p) = - \dot{\gamma}(0)
	\end{equation*}
	because $d(\,\cdot\,,\,\cdot\,)$ is smooth on ${U}\times{V}$. So \eqref{eq:nabla2} yields
	\begin{equation*}
	\liminf_{n\to \infty}\frac{d_K(\exp_p(t_n v)) - d_K(p)}{t_n} \geq - \dot{\gamma}(0) \cdot v.
	\end{equation*}
	In particular,
	\begin{equation*}
	\partial_v^+ d_K(p) \geq \min_{\gamma\in geod(p,K)} - \dot{\gamma}(0) \cdot v.
	\end{equation*}
	With \eqref{eq:upper bound directional derivative}, this concludes the proof of \eqref{e:geometric-main}. Now, \eqref{e:geometric-cor} follows from \eqref{e:geometric-main} and the definition of the generalized gradient \eqref{e:def-gradient} of semiconcave functions. Finally, in order to prove \eqref{e:geometric-two-geodesics}, we consider the vector $v$ that realizes the maximum in \eqref{e:geometric-cor} and we write it as $v=-\alpha\dot\gamma_1(0)-\beta\dot\gamma_2(0)+v_\perp$, where $v_\perp$ is orthogonal to $\dot\gamma_1(0)$ and $\dot\gamma_2(0)$. Then, we have
	$$-v\cdot \dot\gamma_1(0)=\alpha+\beta \dot\gamma_1(0)\cdot \dot\gamma_2(0)\qquad\text{and}\qquad
-v\cdot \dot\gamma_2(0)=\beta+\alpha \dot\gamma_1(0)\cdot \dot\gamma_2(0).$$
In particular,
\begin{equation}\label{e:ghjk}
\min_{\gamma\in geod(p,K)} - \dot{\gamma}(0) \cdot v\le \frac12\Big(-v\cdot \dot\gamma_1(0)-v\cdot \dot\gamma_2(0)\Big)\le \frac12(\alpha+\beta)\Big(1+\dot\gamma_1(0)\cdot \dot\gamma_2(0)\Big).
\end{equation}
Now, using the fact that
$$\alpha^2+\beta^2+2\alpha\beta\dot\gamma_1(0)\cdot \dot\gamma_2(0)\le \|v\|^2=1,$$
we get that
$$(\alpha+\beta)^2\le 1+2\alpha\beta\Big(1-\dot\gamma_1(0)\cdot \dot\gamma_2(0)\Big)\le 1+\frac12(\alpha+\beta)^2\Big(1-\dot\gamma_1(0)\cdot \dot\gamma_2(0)\Big),$$
which implies that
$$(\alpha+\beta)^2\le \frac2{1+\dot\gamma_1(0)\cdot \dot\gamma_2(0)},$$
which, together with \eqref{e:ghjk}, gives \eqref{e:geometric-two-geodesics}.
\end{proof}

As a consequence of \cref{prop:gradient less than 1} and in particular of \eqref{e:geometric-two-geodesics}, we obtain the $\lambda$-cut locus approximates the cut locus in the following sense.

\begin{proposition}\label{prop:approximation}
Suppose that $M$ is a compact Riemannian manifold, the point $b\in M$ is fixed and that $d_b$ is the distance function to $b$. Then, for every $\lambda>0$, $\text{\rm Cut}_b^\lambda(M)\subset \text{\rm Cut}_b(M)$. Moreover, the cut locus $\text{\rm Cut}_b(M)$ is the closure of the union
$\displaystyle \bigcup_{\lambda>0}\text{\rm Cut}_b^\lambda(M)$.
\end{proposition}
\begin{proof}
The inclusion $\text{\rm Cut}_b^\lambda(M)\subset \text{\rm Cut}_b(M)$ follows from the fact that $d_b$ is differentiable and $|\nabla d_b|=1$ outside $ \text{\rm Cut}_b(M)\cup\{b\}$. In order to prove the second claim, we fix a point $p\in \text{\rm Cut}_b(M)$. Then, there is a sequence of points $p_n\in \text{\rm Cut}_b(M)$ for each of which there are at least to different minimizing geodesics from $p_n$ to $b$. Now, from  \eqref{e:geometric-two-geodesics}, we have that $p_n\in  \text{\rm Cut}_b^{\lambda_n}(M)$ for some $\lambda_n>0$. This concludes the proof.
\end{proof}

\section{Regularity of $u_m^d$}\label{sec:regularity}
This section is dedicated to the $C^{1,1}$ regularity of the minimizer $u_m^d$ of \eqref{eq:distance constraint}.
We recall the following result.
\begin{lemma}[Regularization of the obstacle, \cite{generau2019laplacian}]\label{lemma:regularized distance}
  For any $m>0$, there exists a function $\widetilde{d_b}$ which is smooth on $M\setminus\{b\}$, such that
  $$u_m^d \leq \widetilde{d_b} \leq d_b\quad\text{on}\quad M,\qquad\text{and}\qquad \widetilde{d_b}<d_b\quad\text{on}\quad\text{\rm Cut}_b(M).$$
  In particular, $u_m^d$ is also the solution of the obstacle problem
  \begin{equation}
    \min\left\{\int_{M}  \abs{\nabla u}^2 - mu\ :\ u\in H^1(M),\ u\leq \widetilde{d_b}\right\}.
  \end{equation}
\end{lemma}
One could adapt to the manifold framework the regularity theorems for the classical obstacle problem on a euclidean domain and, with the preceding lemma, deduce the regularity of $u_m^d$. Rather than doing that, we will use \cref{lemma:regularized distance} to reduce our problem to a classical obstacle problem on a euclidean domain.
Let us start with the following regularity lemma.

\begin{lemma}[Continuity of $u_d^m$]\label{lemma:regularity}
  For any $m>0$, the function $u_m^d$ is continuous on $M$.
\end{lemma}
\begin{proof}
  We will reduce our problem to a classical obstacle type variational problem on an open subset of $\rn$, by a series of elementary modifications, and apply a classical $W^{2,p}$ regularity theorem.

  From \cref{lemma:regularized distance}, we know that there exists an open set $U\subset M$ and $\varepsilon>0$ such that
  $$\text{\rm Cut}_b(M)\subset U\qquad\text{and}\qquad u_m^d\leq d_b - \varepsilon\quad \text{on}\quad U.$$
  As a consequence, on the set $U$, $u_m^d$ verifies the Euler-Lagrange equation of \eqref{eq:distance constraint}, \textit{i.e} $\Delta u_m^d = -2m$. In particular, it is $C^\infty$ smooth on $U$.
  Let $\Omega\subset M$ be a smooth open set such that
  \begin{equation*}
    U^c \subset \Omega, \quad \partial \Omega \subset U \quad \text{and} \quad \text{\rm Cut}_b(M)\cap \overline{\Omega} = \emptyset.
  \end{equation*}
  As $U^c \subset \Omega$, it suffices to show that $u_m^d$ is continuous on $\Omega$.
  As $\partial \Omega \subset U$, $u_m^d$ is smooth on $\partial \Omega$, so there exists a smooth function $v_m$ on $\overline{\Omega}$ such that $v_m = u_m^d$ on $\partial \Omega$.
  Then, one can check that $u_m^d$ is a solution of the following variational problem:
  \begin{equation*}
    \min\left\{\int_{\Omega}  \abs{\nabla u}^2 - mu\ :\ u\in H^1(\Omega),\ u\leq d_b\ \text{in}\ \Omega,\ u = {v_m}\ \text{on}\ \partial\Omega\right\}.
  \end{equation*}
  As a consequence, $u_m^d-v_m$ is a solution of the following variational problem:
  \begin{equation}
    \label{eq:variational problem on omega}
    \min\left\{\int_{\Omega}  \abs{\nabla v}^2 -(m+\Delta v_m)v\ :\ v\in H^1_0(\Omega),\ v\leq d_b-v_m \ \text{in}\ \Omega\right\}.
  \end{equation}
  Because we have $\text{\rm Cut}_b(M)\cap \overline{\Omega} = \emptyset$, the exponential map at $b$ is a diffeomorphism onto $\Omega$. Let $\phi:\Omega\to \widetilde{\Omega}\subset \rn$ be a normal coordinates chart centered at $b$. Let $g = (g^{ij})$ denotes the metric of $M$ in the coordinates defined by $\phi$, and $\det g$ its determinant. We recall that the Riemannian volume measure is given in coordinates by $\sqrt{\det g}\,\mathrm{d}x$. So we have
  \begin{multline*}
    \int_{\Omega} \left(\abs{\nabla v}^2 -(m+\Delta v_m)v\right)
    = \int_{\widetilde{\Omega}} \left( g^{ij}\partial_i (v \circ \phi^{-1})\partial_j (v \circ \phi^{-1})\sqrt{\det g} -  \big((m+\Delta v_m)\circ \phi^{-1}\big) \big(v\circ \phi^{-1}\big) \sqrt{\det g}\right)\,dx,
  \end{multline*}
  so $(u_m^d-v_m)\circ \phi^{-1}$ is a minimizer of
  \begin{equation}\label{eq:variational problem on omega tilde}
    \min\left\{\int_{\widetilde{\Omega}} \Big(g^{ij}\sqrt{\det g}\,\partial_i w\,\partial_j w - F w\big)\,dx\ :\ w\in H^1_0(\widetilde{\Omega}),\ w\leq \psi \right\},
  \end{equation}
  where we have set $\psi := (d_b-v_m)\circ \phi^{-1}$ and $F := (m+\Delta v_m)\circ \phi^{-1}\sqrt{\det g}$. We want to apply \cite[Theorem 4.32]{troianiello2013obstacle}.
  For this we need to write the above variational problem into a variational inequality.
  Let $w$ be a competitor in \eqref{eq:variational problem on omega tilde}.
  Writing down the minimality of $w_m := (u_m^d-v_m)\circ \phi^{-1}$ against the competitor $w_m + t(w-w_m)$, for $t\in (0,1)$ small, we find that
  \begin{equation*}
    \langle A w_m, w_m-w \rangle \geq \langle F,w_m-w \rangle,
  \end{equation*}
  where $A$ is the elliptic operator defined on $H^1_0(\widetilde{\Omega})$ by $Aw := -\partial_j(g^{ij}\sqrt{\det g}\,\partial_i w)$.
  From there, we can apply \cite[Theorem 4.32]{troianiello2013obstacle} to deduce that, for any $p<n$, if $A \psi\wedge F \in L^p(\widetilde{\Omega})$, then $A w_m \in L^p(\widetilde{\Omega})$. To check that $A \psi\wedge F \in L^p(\widetilde{\Omega})$, it is enough to check that $A (d_b \circ \phi^{-1}) \in L^p(\widetilde{\Omega})$.
  As $d_b$ is smooth except at $b$, it is enough to check that $(A (d_b \circ \phi^{-1}))^p$ is integrable at $0$.
  But this is a consequence of the fact that $-\Delta d_b \circ \phi^{-1}= \frac{1}{\sqrt{\det g}}A(d_b \circ \phi^{-1})$, and \cref{lemma:laplacian distance} below, from which we deduce that $A (d_b \circ \phi^{-1})(x)$ is equivalent to $\frac{n-1}{\abs{x}}$ when $x$ goes to $0$.
  Therefore, for $p<n$, $(A (d_b \circ \phi^{-1}))^p$ is integrable at $0$, and so $Aw_m \in L^p(\widetilde{\Omega})$. By elliptic regularity, this implies $w_m \in W^{2,p}(\widetilde{\Omega})$, for any $p<n$. By the Sobolev embeddings, $w_m$ is then continuous on $\widetilde{\Omega}$, and so $u_m^d$ is continuous on $\Omega$. This concludes the proof.
\end{proof}

We can now define the set $E_m^d := \{u_m^d<d_b\}$, for any $m>0$. It is an open subset of $M$, on which $u_m^d$ solves the equation $\Delta u_m^d = -2m$.
We can now prove the following lemma.

\begin{lemma}\label{lemma:u=d around b}
  For any $m>0$, we have $u_m^d=d_b$ in a neighborhood of $b$.
\end{lemma}
\begin{proof}
  Let us assume that we have constructed a $C^1$ function $v$ on $\overline{B}_R(b)$ for some $R>0$, such that

  \begin{numcases}{}
    v\leq d_b & in  $B_R(b)$, \label{eq:v leq d_b}\\
    v=d_b  & in  $B_\eps(b)$  for some  $\varepsilon\in(0,R)$,\label{eq:v = d_b}\\
    v<0  & in $\partial B_R(b)$, \label{eq:v<0}\\
    \Delta v \geq -m  & in  $B_R(b)$ in the distributional sense. \label{eq:laplacian constraint}
  \end{numcases}
  We will then show that we have $u_m^d \geq v$. The construction of $v$ is postponed to the end of the proof. From \cref{lemma:regularity}, we know that the function $v-u_m^d$ is continuous. Let us first assume that $v-u_m^d$ attains a positive maximum at a point $x \in \overline{B}_R(b)$. We have
  \[
  0<v(x)-u_m^d(x)\leq d_b(x)-u_m^d(x),
  \]
  so $x\in E_m^d$. Moreover, we have $u_m^d \geq 0$ since $\max(u_m^d,0)$ is a better competitor than $u_m^d$ in \eqref{eq:distance constraint}, so
  \[v-u_m^d\leq v<0 \quad \text{on} \quad \partial B_R(b),\]
  and so $x\in B_R(b)$. Hence the function $v-u_m^d$ attains a positive maximum inside the open set $E_m^d \cap B_R(b)$, but its Laplacian verifies in the distributional sense:
  \begin{equation}
    \Delta(v-u_m^d) = \Delta v +m\geq 0,
  \end{equation}
  which yields a contradiction by the maximum principle. Then, the maximum of $v-u_m^d$ on $\overline{B}_R(b)$ is non-positive, and we get
  \[u_m^d \geq v = d_b \quad \text{in} \quad B_\eps(b),\]
  which concludes the proof.

  Let us now construct the function $v$ that was used above. Let $R>0$ be small enough so that $\overline{B}_R(b)$ is contained in a normal neighborhood of $b$. In polar coordinates around $b$, we define $v$ as a radial function. For $\varepsilon>0$ to be chosen small enough later, let $f:[0,R]\to[0,\infty)$ be the $C^1$ function such that
  \begin{numcases}{}
    f(r)=r \quad &\text{if}\quad $r\leq \varepsilon$, \nonumber\\
    f''(r)+\frac{n-1}r f'(r)=-\frac{m}{2} &\text{if}\quad $r> \varepsilon$. \label{eq:spherical laplacian}
  \end{numcases}
  If $n=2$, the unique $C^1$ solution to this system is given by:
  \begin{numcases}{}
    f(r)=r & \text{if} \quad $r\leq \varepsilon$, \nonumber\\
    f(r)=\varepsilon +\frac{m}{8}\left(\varepsilon^2-r^2\right)+\left(\varepsilon+\frac{m}{4}\varepsilon^2\right)\ln(\frac{r}{\varepsilon}) & if \quad $r> \varepsilon$. \label{eq:f(r)2}
  \end{numcases}
  If $n\ge 3$, then the solution is
  \begin{numcases}{}
    f(r)=r & if \quad $r\leq \varepsilon$, \nonumber\\
    f(r)=\varepsilon +\frac{m}{4n}\left(\varepsilon^2-r^2\right) \nonumber \\
    \quad\quad\quad+\left(\varepsilon^{n-1}+\frac{m}{2n}\varepsilon^n\right) \frac{1}{n-2} \left(\frac{1}{\varepsilon^{n-2}}-\frac{1}{r^{n-2}}\right) & if \quad $r> \varepsilon$. \label{eq:f(r)3}
  \end{numcases}
  Then, we set in standard polar coordinates $v(x) = f(r)$ for $x \in B_R(b)$. For $r \leq \varepsilon$, the constraint \eqref{eq:v = d_b} is verified by definition. (For $r>\varepsilon$, we chose $f$ so that $\Delta v$ is small, but still bigger than $-m$.)

  Let us show that \eqref{eq:v leq d_b} holds.
  Let us set $g(r):=f(r)-r$ and prove that $g\leq0$. We have $g(r)=0$ for $r\leq \varepsilon$ so it is sufficient to prove that $g'(r)\leq0$ for $r\geq\varepsilon$. But, as $f$ verifies \eqref{eq:spherical laplacian}, $g$ verifies
  \begin{equation*}
    g''+\frac{n-1}{r}g'=-m-\frac{n-1}{r} \quad \text{for} \quad r\geq \varepsilon.
  \end{equation*}
  In particular, whenever $g'(r)=0$, we have $g''(r)<0$. This implies $g'(r)\leq0$ for $r\geq\varepsilon$, and so \eqref{eq:v leq d_b} is verified.

  Now let us show that \eqref{eq:laplacian constraint} holds if $R$ has been taken small enough. We use the following expression of the Laplacian in coordinates:
  \begin{equation*}
    \Delta v = \frac{1}{\sqrt{\det g}}\partial_i\left(\sqrt{\det g}g^{ij}\partial_j v\right),
  \end{equation*}
  where $g = (g^{ij})$ is the metric of the manifold $M$, and $\det g$ its determinant. We apply this formula to polar coordinates to find that, on ${B_R(b)\setminus \overline{B}_\eps(b)}$, we have in the classical sense
  \begin{align}
    \Delta v = \frac{1}{\sqrt{\det g}}\partial_r\big(\sqrt{\det g}f'(r)\big) &= f'' + \frac{\partial_r \det g}{2\det g}f' \nonumber \\
    &= f'' + \frac{n-1}{r}f'+ \left(\frac{\partial_r \det g}{2\det g}-\frac{n-1}{r}\right)f' \nonumber \\
    &= -\frac{m}{2} + \left(\frac{\partial_r \det g}{2\det g}-\frac{n-1}{r}\right)f'. \label{eq:laplacian v}
  \end{align}
  Note that by applying the Laplacian formula in polar coordinates to the distance function $d_b(x) = r$, we find that
  \begin{equation}
    \Delta d_b = \frac{\partial_r \det g}{2\det g}. \label{eq:laplacian distance}
  \end{equation}
  Because of \cref{lemma:laplacian distance}, we also have
  \begin{equation}
    \Delta d_b(x) = \frac{n-1}{r} + o(1). \nonumber\label{eq:laplacian distance bis}
  \end{equation}
  With \eqref{eq:laplacian v} and \eqref{eq:laplacian distance}, this last equation yields in the classical sense
  \begin{equation}
    \Delta v = -\frac{m}{2}+o(1)f'(r)\quad \text{on} \quad B_R(b)\setminus \overline{B}_\eps(b). \label{eq:laplacian v bis}
  \end{equation}
  Moreover, it is clear from the following expression that $f'$ is bounded on $[\varepsilon,R]$, by a constant independent of $R$, as long as we choose $R\leq1$:
  \begin{equation*}
    f'(r) = -\frac{m}{n}r+\left(\varepsilon^{n-1}+\frac{m}{n}\varepsilon^{n}\right)\frac{1}{r^{n-1}}\quad \text{for}\quad \varepsilon \leq r \leq R.
  \end{equation*}
  Hence from \eqref{eq:laplacian v bis} we see that by taking $R$ small enough (independently of $\varepsilon$), we can ensure that
  \begin{equation}
    \Delta v \geq -m \quad \text{on} \quad B_R(b)\setminus \overline{B}_\eps(b). \nonumber
  \end{equation}
  But from \eqref{eq:laplacian distance bis}, we see that the above is also true on $B_\eps(b)$ if $\varepsilon$ is small enough. Thus the function $v$ is $C^1$ on $B_R(b)$ and verifies $\Delta v(x) \geq -m$ when $x\notin \partial B_\eps(b)$, hence \eqref{eq:laplacian constraint} holds.
  It is also clear from \eqref{eq:f(r)2} and \eqref{eq:f(r)3} that the constraint \eqref{eq:v<0} is verified if $\varepsilon$ is taken small enough. This concludes the proof.
\end{proof}

We can now prove the $C^{1,1}$ regularity of $u_m^d$.

\begin{proposition}\label{prop:regularity}
  For any $\varepsilon>0$, the function $u_m^d$ belongs to $C^{1,1}(M\setminus B_\eps(b))$.
\end{proposition}
\begin{proof}
  We reproduce the proof of \cref{lemma:regularity}, but we replace the open set $\Omega$ with $\widehat{\Omega}:= \Omega\setminus B_\eps(b)$, and the function $v_m$ with a function $\widehat{v_m}$ that is smooth and such that $w_m = u_m$ on $\partial\widehat{\Omega}$.
  We know that such a function exists because $u_m$ is smooth on $\partial B_\eps(b)$ for $\varepsilon$ small enough, as it can be seen from \cref{lemma:u=d around b}.
  This way, we can apply the stronger $W^{2,\infty}$ regularity result for the obstacle problem \cite[Theorem 4.38]{troianiello2013obstacle}, since $d_b$ is smooth on $\widehat{\Omega}$. We get that $u_m^d$ belongs to $W^{2,\infty} = C^{1,1}(\widehat{\Omega})$.
  As $u_m^d$ is smooth on $E_m^d$ and $\partial \widehat{\Omega} \subset E_m^d$, then $u_m^d$ is $C^{1,1}$ on $\widehat{\Omega} \cup E_m = M\setminus B_\eps(b)$.
\end{proof}

We end this section with the following computational lemma, which we used in the proof of \cref{lemma:u=d around b}.
\begin{lemma}\label{lemma:laplacian distance}
  We have
  \begin{equation}\label{e:laplacian_d_b}
    \Delta d_b(p) \underset{p \to b}{=}  \frac{n-1}{d_b(p)} + o(1).
  \end{equation}
\end{lemma}
\begin{proof}
  We compute $\Delta d_b$ in normal coordinates centered at $b$. Let $g = (g^{ij})$ be the metric of $M$ in these coordinates. We have
  \begin{equation*}
    \Delta d_b (x) = \frac{1}{\sqrt{\det g}}\,\partial_i \left(\sqrt{\det g}\,g^{ij} \partial_j d_b\right)(x).
  \end{equation*}
  In normal coordinates, the metric is euclidean up to order $1$ as $x$ goes to $0$. So we have
  \begin{equation*}
    g^{ij}(x) = \delta^{ij} + o(x), \quad \partial_i \left(\sqrt{\det g}\,g^{ij}\right)(x) = o(1)\quad \text{and} \quad \frac{1}{\sqrt{\det g}} = 1 + o(x).
  \end{equation*}
Moreover, in normal coordinates, we have $d_b(x) = \abs{x}$, and so
  \begin{equation*}
    \delta^{ij}\partial_{ij}d_b(x) = \frac{n-1}{\abs{x}},
  \end{equation*}
  which gives precisely \eqref{e:laplacian_d_b}.
\end{proof}


\section{Equivalence of the two constraints}\label{sec:equivalence constraints}

\begin{proof}[Proof of \cref{prop:equivalence constraints}]
 As above, we denote by $u_m^d$ the minimizer of \eqref{eq:distance constraint}. In order to show that $u_m^d$ solves \eqref{eq:gradient constraint manifold}, it is sufficient to show that $u_m^d$ is an admissible competitor in \eqref{eq:gradient constraint manifold}, that is, $|u_m^d|\le 1$ on $M$. Recall that the function $u_m^d$ is $C^1$ except at $b$, by  \cref{prop:regularity}.

 First, suppose that $x\neq b$ is in the contact set $P_m^d:=\{u_m^d=d_b\}$. By \cref{lemma:regularized distance}, we have $x\notin \text{\rm Cut}_b(M)$, and so the distance function $d_b$ is differentiable at $x$. It is a simple consequence of the constraint $u_m^d \leq d_b$ and the equality $u_m^d(x)=d_b(x)$ that we have $\nabla u_m^d(x)=\nabla d_b(x)$. The desired inequality $\abs{\nabla u_m^d(x)}\leq 1$ follows.

  In the non-contact set $E_m^d=\{u_m^d < d_b\}$, the function $u_m^d$ solves the PDE
  \begin{equation}
    \label{eq:EL}\Delta u_m^d = -2m.
  \end{equation}
  In particular it is smooth, and we may apply the Bochner-Weitzenböck formula:
    \begin{equation}
        \label{eq:bochner}
        \Delta \left(\abs{\nabla u_m^d}^2\right) = 2\mathrm{Ric}(\nabla u_m^d,\nabla u_m^d)+2\abs{D^2u_m^d}^2+2(\nabla \Delta u_m^d,\nabla u_m^d),
      \end{equation}
      where $\mathrm{Ric}$ denotes the Ricci curvature tensor on the manifold $M$ and $D^2 u_m^d$ is the second covariant derivative of $u_m^d$. The last term is $0$ because of \eqref{eq:EL}. As for the second term, we have:
    \begin{align}
      \abs{D^2u_m^d}^2 &\geq \frac1n\left(\mathrm{Trace}(D^2u_m^d)\right)^2=  4\frac{m^2}{n}\,,
    \end{align}
    where the last inequality is due to \eqref{eq:EL}.
    As the manifold $M$ is compact, there exists a constant $K>0$ (depending on $M$ only) such that the Ricci curvature is bounded from below by $-K$.
    In the end, \eqref{eq:bochner} yields
    \begin{equation}\label{eq:bochner applied}
      \Delta \left(\abs{\nabla u_m^d}^2\right) + 2K \abs{\nabla u_m^d}^2 \geq \frac{8}{n}m^2.
    \end{equation}
    Now notice that by \eqref{eq:EL},
    \begin{align}
      \Delta\left((u_m^d)^2\right)&=2u_m^d\Delta u_m^d+2\abs{\nabla u_m^d}^2  =-4mu_m^d+2\abs{\nabla u_m^d}^2, \nonumber
    \end{align}
    so \eqref{eq:bochner applied} gives
    \begin{align*}
      \Delta\left(\abs{\nabla u_m^d}^2+K(u_m^d)^2\right)= \frac8n m^2 - 4Km\, u_m^d & \geq \frac8n m^2 - 4Km\, d_b  \geq \frac8n m^2 - 4Km\, \mathrm{diam}(M)
    \end{align*}
    Thus, if $m\geq \frac{n}{2} K \mathrm{diam}(M)$, the function $\abs{\nabla u_m^d}^2+K(u_m^d)^2$ is subharmonic in the non-contact set $E_m^d$. From \cref{lemma:u=d around b}, we have $\overline{E_m^d}\subset M\setminus\{b\}$, and with \cref{prop:regularity},
    we get that the function $\abs{\nabla u_m^d}^2+K(u_m^d)^2$ is continuous on $\overline{E_m^d}\subset M\setminus\{b\}$. Therefore we may apply the maximum principle to get
    \begin{align}
      \abs{\nabla u_m^d}^2  \leq \abs{\nabla u_m^d}^2+K(u_m^d)^2& \leq \sup\limits_{\partial E_m^d}\left(\abs{\nabla u_m^d}^2+K(u_m^d)^2\right) = 1+K\sup\limits_{\partial E_m^d} (u_m^d)^2 \nonumber \\
      &\leq 1+K\sup\limits_{\partial E_m^d} (d_b)^2 \leq 1+K\mathrm{diam}(M)^2\nonumber
    \end{align}
    With \eqref{eq:bochner applied}, this last inequality gives
    \begin{equation}
      \Delta \left(\abs{\nabla u_m^d}^2\right) \geq \frac{8}{n}m^2 - 2K(1+K\mathrm{diam}(M)^2)\nonumber
    \end{equation}
    Thus, whenever the right-hand side is nonnegative, the maximum principle applied to the function $\abs{\nabla u_m^d}^2$ on the open set $E_m^d$ implies that $\abs{\nabla u_m^d}^2 < 1$ on this set. This concludes the proof.
\end{proof}


\section{Convergence of the non-contact set}\label{sec:convergence non-contact set}
In this section we show that the non-contact set $E_{m}^d=\{u_m^d< d_b\}$ (which coincides with $E_m$, for $m$ large enough, as we showed in the previous section) Hausdorff-converges to $\text{\rm Cut}_b(M)$.
\begin{lemma}
  \label{lemma:uniform convergence}
  We have $\displaystyle\|d_b-u_m^d\|_{L^\infty(M)}\leq \frac{C}{m}$, for some positive constant $C$ depending on $M$ only.
\end{lemma}
\begin{proof}
  We only need to prove the proposition for $m$ large enough. Therefore, thanks to \cref{prop:equivalence constraints}, we will assume that $m$ is large enough so that $\abs{\nabla u_m^d}\leq1$. We only need to show the estimate on $E_m^d$ since outside this set, $u_m^d$ and $d_b$ are the same.
  We will show that for $m$ large enough, we have
  \begin{equation}\label{eq:elastic set thin}
    \forall \overline{p}\in E_m^d, \; \exists p\in (E_m^d)^c\quad \text{such that}\quad d(p,\overline{p})< 5n/m.
  \end{equation}
  This will conclude the proof since by the $1$-Lipschitzianity of $u_m^d$ and $d_b$, we then have
  \begin{align}
    \abs{d_b(\overline{p})-u_m^d(\overline{p})}
    &\leq \abs{d_b(p)-u_m^d(p)} + 2 d(\overline{p},p) = 0 + 2 d(\overline{p},p)\leq \frac{10n}{m}, \nonumber
  \end{align}
  which is what we need.
  In order to prove \eqref{eq:elastic set thin}, we argue by contradiction and assume that $B_{5n/m}(\overline{p}) \subset E_m^d$.
  We want to apply the maximum principle  to the function $v$ defined on $B_{5n/m}(\overline{p})$ by the following formula
  \[v(p) := u_m^d(p) - \inf\limits_{\partial B_{\frac{5n}{m}}(\overline{p})} u_m^d + \frac{m}{2n}\left(d_{\overline{p}}(p)^2-\left(\frac{5n}{m}\right)^2\right).\]
  For any $p\in B_{5n/m}(\overline{p})$, we have $\Delta u_m^d (p) = -2m$ because we have assumed $B_{5n/m}(\overline{p}) \subset E_m^d$. To estimate the Laplacian of $d_{\overline{p}}^2$, we use some normal coordinates $(x^i)$ centered at $\overline{p}$.
  In these coordinates, the metric is euclidean up to order $1$, uniformly in $\overline{p}$ since $M$ is compact, and $d_{\overline{p}}(x) = \abs{x}$ (see \cref{lemma:laplacian distance}).
  We get that for $m$ large enough, independently of $\overline{p}$,
  \[\forall p \in B_{5n/m}(\overline{p}), \quad \Delta d_{\overline{p}}^2(p) \leq 2 (2n).\]
  All in all, we obtain on $B_{5n/m}(\overline{p}) \subset E_m^d$,
  \[\Delta v \leq -2m+\frac{m}{2n}2(2n) = 0.\]
  So we can apply the maximum principle to $v$ to get
  \[ v(\overline{p})\geq \inf\limits_{\partial B_{\frac{5n}{m}}(\overline{p})} v,  \]
  \textit{i.e.}
  \begin{equation}\label{eq:quadratic growth}
    u_m^d(\overline{p}) - \inf\limits_{\partial B_{\frac{5n}{m}}(\overline{p})} u_m^d - \frac{m}{4n}\left(\frac{5n}{m}\right)^2 \geq 0.
  \end{equation}
  As we have taken $m$ large enough so that $\abs{\nabla u_m^d}\leq 1$, we also have
  \[u_m^d(\overline{p}) - \inf\limits_{\partial B_{\frac{5n}{m}}(\overline{p})} u_m^d\leq \frac{5n}{m} < \frac{m}{4n}\left(\frac{5n}{m}\right)^2,\]
  which contradicts the estimate \eqref{eq:quadratic growth}. This concludes the proof.
\end{proof}
\begin{proposition}[Monotonicity of $u_m^d$ and $E_m^d$, and convergence of $E_m^d$]\label{thm:non-contact set converges to cut locus}
  For any $m>m'>0$, we have
  $$u_{m'}^d\le u_{m}^d\le d_b\qquad\text{and}\qquad \text{\rm Cut}_b(M) \subset E_m^d\subset E_{m'}^d.$$
  Moreover,
  \begin{equation*}
    E_{m}^d \tend{m}{\infty}\text{\rm Cut}_b(M)\quad\text{in the Hausdorff sense.}
  \end{equation*}

\end{proposition}
\begin{proof}
  The fact that, for any $m>0$, $\text{\rm Cut}_b(M)\subset E_{m}^d$, is a direct consequence of  \cref{lemma:regularized distance}. Let us prove the second inclusion. For $m>m'>0$, note that by the respective  minimality of $u_m^d$ and $u_{m'}^d$, we have
  \begin{align*}
    \int_M \abs{\nabla \max(u_{m'}^d,u_m^d)}^2-m \int_M \max(u_{m'}^d,u_m^d)&\geq \int_M \abs{\nabla u_m^d}^2-m \int_M u_m^d,
    \\ \text{and} \quad\int_M \abs{\nabla \min(u_{m'}^d,u_m^d)}^2-m' \int_M \min(u_{m'}^d,u_m^d)&\geq \int_M \abs{\nabla u_{m'}^d}^2-m' \int_M u_{m'}^d.
  \end{align*}
	Using the formulas
  \begin{align*}
    \nabla \max(u_{m'}^d,u_m^d) &= \nabla u_{m'}^d \1_{\{u_{m'}^d>u_m^d\}} + \nabla u_{m}^d \1_{\{u_{m'}^d\leq u_m^d\}}, \\
    \nabla \min(u_{m'}^d,u_m^d) &= \nabla u_{m}^d \1_{\{u_{m'}^d>u_m^d\}} + \nabla u_{m'}^d \1_{\{u_{m'}^d\leq u_m^d\}},
  \end{align*}
  we obtain
  \begin{align*}
    \int_{\{u_{m'}^d>u_m^d\}} \left(\abs{\nabla u_{m'}^d}^2 - \abs{\nabla u_m^d}^2\right)
    &\geq -m \int_{\{u_{m'}^d>u_m^d\}} \left(u_m^d - u_{m'}^d\right),
    \\ \text{and}\quad\int_{\{u_{m'}^d>u_m^d\}} \left(\abs{\nabla u_m^d}^2-\abs{\nabla u_{m'}^d}^2\right)
    &\geq -m'\int_{\{u_{m'}^d>u_m^d\}} \left(u_{m'}^d-u_m^d\right).
  \end{align*}
  Summing these two inequalities, we get
  \begin{equation*}
    0\geq(m-m')\int_{\{u_{m'}^d>u_m^d\}} \left(u_{m'}^d-u_m^d\right),
  \end{equation*}
  and so $u_m^d\geq u_{m'}^d$. In particular, $E_m^d\subset E_{m'}^d$.

  We are left to show the Hausdorff convergence in $E_m^d$ to $\text{\rm Cut}_b(M)$. Given $\varepsilon>0$, let us set
  \[\Omega_\varepsilon := \big\{x\in M : d(x,\text{\rm Cut}_b(M))>\varepsilon \big\}.\]
  We will show that for $m$ large enough we have $E_{m}^d\subset (\Omega_{2\varepsilon})^c$, which will conclude the proof. Let $\phi:M\to \R$ be a function such that $\phi\leq d_b$ on $M$, $\phi = d_b$ on $\Omega_{2\varepsilon}$, $\phi < d_b$ on $\partial \Omega_{\varepsilon}$, and $\phi$ is smooth on $M$ except at $b$.
  We want to apply the maximum principle to the function $\phi-u_m^d$ on $E_m^d\cap \Omega_\varepsilon$.
  We have
  \[\Delta(\phi - u_m^d) = \Delta\phi+2m\quad\text{on}\quad E_m^d\cap \Omega_\varepsilon,\]
  so for $m$ large enough the function $\phi-u_m^d$ is subharmonic on $E_m^d\cap \Omega_\varepsilon$. On $\partial \Omega_\varepsilon$, we have $\phi< d_b$ and $u_m^d$ converges uniformly to $d_b$ as $m$ tends to $+\infty$ (\cref{lemma:uniform convergence}) so $\phi-u_m^d \leq 0$, for $m$ large enough.
  On $\partial E_m^d$, we have $\phi-u_m^d = \phi -d_b \leq 0$. Thus the maximum principle implies that for $m$ large enough, we have $\phi - u_m^d \leq 0$ on $E_m^d\cap \Omega_\varepsilon$. As $\phi = d_b$ on $\Omega_{2\varepsilon}$, we get $u_m^d \geq d_b$ on $E_m^d\cap \Omega_{2\varepsilon}$.
  Since by definition we have $u_m^d < d_b$ on $E_m^d$, we get $E_{m}^d\subset (\Omega_{2\varepsilon})^c$, which concludes the proof.
\end{proof}


\section{Semiconcavity}\label{sec:semiconcavity}
This section is dedicated to the semiconcavity of the solutions to the obstacle problems \eqref{eq:distance constraint} and \eqref{eq:elastoplasticequivalence}. The key result is \cref{prop:semiconcavity}, which applies to both \cref{t:main-main} and \cref{thm:omega}.
\smallskip

In the case of \cref{thm:omega}, we have $\mathring{M}=\Omega$ and $\partial M=\partial\Omega$.

\begin{proposition}\label{prop:semiconcavity}
	Let $M = \mathring{M}\sqcup \partial M$ be a smooth compact Riemanniannian manifold, with  (possibly empty) boundary $\partial M$. Suppose that for some constants $L>0$ and $C>0$, we are given the following:
	\begin{enumerate}[\rm (a)]
	\item a function $d:\mathring{M}\to\R$, which is bounded and $C$-semiconcave on $\mathring{M}$;
	\item a family of functions $u_m:\mathring{M}\to\R^n$, for $m>0$, such that:
	\begin{enumerate}
	\item[\quad\rm (b.1)] for every $m>0$, $u_m\le d$ on $\mathring{M}$;
	\item[\quad\rm (b.2)] for every $m>0$, $u_m$ is $L$-Lipschitz on $\mathring{M}$;
	\item[\quad\rm (b.3)] on the set $E_m:=\{u_m<d\}$, $u_m$ is $C^\infty$ smooth and
	$$-\Delta u_m=2m\quad\text{in}\quad E_m\,;$$
	\item[\quad\rm (b.4)] $E_m$ is precompact in $\mathring{M}$;
	\item[\quad\rm (b.5)] for every $\eta>0$, for every $m>0$, there is a neighborhood $\mathcal N_{\eta,m}$ of $\partial  E_m$ in $\mathring{M}$ such that
	$$D^2u_m\le \left(C+\eta \right)Id\quad\text{in}\quad E_m\cap \mathcal N_{\eta,m}.$$
	\end{enumerate}
	\end{enumerate}
	Then, for every $\eta >0$, there exists $m_0>0$  such that
	\begin{center}
	$u_m$ is $(C+\eta)$-semiconcave on $\mathring{M}$,  for every $m\ge m_0$.
	 \end{center}
\end{proposition}

\noindent{\bf Application to \cref{thm:omega}.} In order to apply \cref{prop:semiconcavity} to \cref{thm:omega}, we take $\mathring{M}=\Omega$ and $\partial M=\partial\Omega$.
The function $d$ is the distance function $d_{\partial\Omega}$ to the boundary of $\Omega$, while $u_m$ is the solution $v_m$ of \eqref{eq:elastoplasticequivalence} (thus, the Lipschitz constant from (b.2) is $L=1$), which means that the conditions (b.1), (b.2) and (b.3) are fullfilled. When $\Omega$ is $C^2$ regular, the set $\overline{\mathcal M(\Omega)}$ is contained in $\Omega$.
Now, as the elastic sets $\{u_m<d\}$ Hausdorff-converge to $\overline{\mathcal M(\Omega)}$ (see \cite{caffarelli_friedman_elastoplastic}) we get that, for large $m$, $u_m$ coincides with $d$ in a neighborhood of $\partial \Omega$.
Thus, (b.4) is fullfilled. As $\Omega$ is $C^2$, the function $d_{\partial \Omega}$ is known to be $C$-semiconcave in $\Omega$ for some $C>0$ (see \cite[(iii) of Proposition 2.2.2]{cannarsa2004semiconcave}), so (a) is fulfilled. Finally, condition (b.5) is a consequence of \cite[Chapter 2, Theorem 3.8]{friedman1982variational}.
Thus, there exists a constant $C>0$ such that for $m$ big enough, $v_m$ is $C$-semiconcave in $\Omega$.
\smallskip

\noindent{\bf Application to \cref{t:main-main} \ref{item:semiconcavity}.} In the case of \cref{t:main-main}, we take $\mathring{M}=M\setminus \overline B_\rho(x_0)$ and $\partial M=\partial B_\rho(b)$, where $B_\rho(b)$ is a small geodesic ball centered at the base point $b$. The function $d$ is the distance function $d_{b}$ to the base point, while $u_m$ is the solution of \eqref{eq:distance constraint}.
The semiconcavity of the distance function $d$ in $M\setminus B_\rho(b)$ was proved in \cite{mantegazza_mennucci_2003}, see Proposition \ref{prop:db semiconcave}.
By \cref{prop:equivalence constraints}, for large $m$, the problems \eqref{eq:distance constraint} and \eqref{eq:gradient constraint manifold} are equivalent and so we can take $L=1$ in (b.2), and we also have that (b.1) are (b.3) are fullfilled.
Next, we notice that by \cref{lemma:u=d around b} we have that $u_m=d$ in a neighborhood of $b$, which proves (b.4) by choosing the radius $\rho$ small enough. Finally, in \cref{p:bound-on-D2-near-the-boundary} we will prove that also the condition (b.5) is fullfilled.

\begin{proof}[\bf Proof of \cref{prop:semiconcavity}]
First, we notice that by dividing all the functions by $L$, we can assume that $L=1$. Let $\eta>0$. As in \cref{d:semiconcavity}, for $a,b\in\R$ and $\lambda \in (0,1)$, we will use the notation
  \begin{equation*}
    \lambda_{ab}:=(1-\lambda)a+\lambda b.
  \end{equation*}
  For any unit speed geodesic $\gamma:[a,b]\to \mathring{M}$, $\lambda\in(0,1)$ and $v$ a function on $\mathring{M}$, let us define
  \begin{align}
    \nonumber c(\gamma,\lambda,v)
    &:= \lambda(1-\lambda)(C+\eta)(b-a)^2  -\Big((1-\lambda)v(\gamma(a)) + \lambda v(\gamma(b))-v(\gamma(\lambda_{ab}))\Big).
  \end{align}
  We aim to show the following:
  \begin{equation}
    \label{eq:semiconcave}
		\inf\limits_{\gamma,\lambda} c(\gamma,\lambda,u_m)\geq 0,
  \end{equation}
  where the infimum is taken over unit speed geodesics defined over finite intervals. Let us argue by contradiction and assume that \eqref{eq:semiconcave} does not hold.

	Let us show that we may assume that the infimum is actually taken over unit speed geodesics $\gamma : [a,b]\to \mathring{M}$ such that
	\begin{equation}
    \gamma\big((a,b)\big)\subset E_m = \{u_m < d\}. \label{eq:gamma non-contact}
  \end{equation}
  Let $\gamma : [a,b]\to \mathring{M}$ be a unit speed geodesic, and $\lambda \in (0,1)$, such that $c(\gamma,\lambda,u_m)<0$.
	Let us assume that $\gamma$ does not verify \eqref{eq:gamma non-contact}. We will build a geodesic $\widehat{\gamma}$ that does verify \eqref{eq:gamma non-contact}, and $\widehat{\lambda}\in (0,1)$, such that
	\begin{equation*}
		c(\widehat{\gamma},\widehat{\lambda},u_m)<c(\gamma,\lambda,u_m).
	\end{equation*}
	First, notice that if $\gamma(\lambda_{ab}) \notin E_m$, then we have $u_m(\gamma(\lambda_{ab})) = d(\gamma(\lambda_{ab}))$, $u_m(\gamma(a)) \leq d(\gamma(a))$ and $u_m(\gamma(b)) \leq d(\gamma(b))$, and so
	\begin{equation}
    c(\gamma,\lambda,u_m)\geq c(\gamma,\lambda,d)>0, \label{eq:c positive Chapter cut locus}
  \end{equation}
  where the last inequality comes from the $C$-semiconcavity of $d$. This is contradictory, so $\gamma(\lambda_{ab}) \in E_m$.
	As $\gamma$ does not verify \eqref{eq:gamma non-contact}, there exists $t \in (0,\lambda_{ab})\cup(\lambda_{ab},1)$, such that $\gamma(t_{ab})\notin E_m$. Up to reparametrization of $\gamma$, we may assume that $t\in (0,\lambda_{ab})$. We can define
	\begin{equation*}
		\mu := \min\left\{s\in (0,\lambda) : \forall r \in (s,\lambda), \quad \gamma(r_{ab}) \in E_m\right\}.
	\end{equation*}
	We have $\gamma(\mu_{ab}) \notin E_m$, and $\gamma((\mu_{ab},\lambda_{ab}]) \subset E_m$.
	Figure 2 may help justify intuitively the following construction. Let $\widetilde{\lambda} \in (0,1)$ be such that
  \begin{equation}
    \widetilde{\lambda}_{\mu_{ab}b}=\lambda_{ab}. \label{eq:definition nu Chapter cut locus}
  \end{equation}
    \begin{center}
      \begin{tikzpicture}
        \draw[->] ({-1},{0}) -- ({8},{0});
        \draw (8,0) node[below right] {$t$};
        \draw[->] ({0},{-0.5}) -- ({0},{2.5});
        \draw (0,2) node[above left] {$f(t)$};
        \draw (1,1.2) node {$\bullet$};
        \draw (3,0.5) node {$\bullet$};
        \draw (5,2) node {$\bullet$};
        \draw (7,1.2) node {$\bullet$};
        \draw[dashed] (1,1.2) -- (1,0);
        \draw[dashed] (3,0.5) -- (3,0);
        \draw[dashed] (5,2) -- (5,0);
        \draw[dashed] (7,1.2) -- (7,0);
        \draw (1,0) node[below] {$a$};
        \draw (3,0) node[below] {$\mu_{ab}$};
        \draw (5,0) node[below] {$\lambda_{ab} = \widetilde{\lambda}_{\mu_{ab}b}$};
        \draw (7,0) node[below] {$b$};
      \end{tikzpicture}

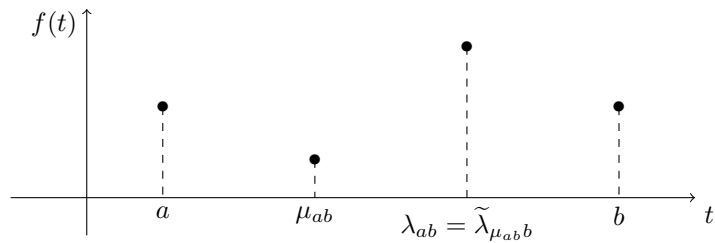
      {\sc Figure 2. \it Construction of $\widetilde{\gamma}$ and $\widetilde{\lambda}$.}
    \end{center}
 \noindent Let $\widetilde{\gamma}$ be the unit speed geodesic defined by $\widetilde{\gamma}:=\gamma_{|[\mu_{ab},b]}$. Let us set $f(t):=(C+\eta)t^2-u_m(\gamma(t))$. Then
  \begin{align}
    c(\widetilde{\gamma},\widetilde{\lambda},u_m)
    &=(1-\widetilde{\lambda})f(\mu_{ab})+\widetilde{\lambda} f(b)-f(\widetilde{\lambda}_{\mu_{ab}b}) \nonumber \\
    &=(1-\widetilde{\lambda})f(\mu_{ab})+\widetilde{\lambda} f(b)-f(\lambda_{ab}) \nonumber \\
    &=c(\gamma,\lambda,u_m)-(1-\lambda)f(a) + (\widetilde{\lambda}-\lambda)f(b) + (1-\widetilde{\lambda})f(\mu_{ab}). \label{eq:contradiction on c Chapter cut locus}
  \end{align}
  Now after some elementary calculations, \eqref{eq:definition nu Chapter cut locus} translates into
  \begin{equation*}
    \begin{cases}
      1-\lambda&=(1-\widetilde{\lambda})(1-\mu), \\
      \widetilde{\lambda} - \lambda&=-(1-\widetilde{\lambda})\mu,
    \end{cases}
  \end{equation*}
  so \eqref{eq:contradiction on c Chapter cut locus} becomes
  \begin{align}
    c(\widetilde{\gamma},\widetilde{\lambda},u_m)
    &=c(\gamma,\lambda,u_m)-(1-\widetilde{\lambda})\big((1-\mu)f(a)+\mu f(b)-f(\mu_{ab})\big) \nonumber \\
    &=c(\gamma,\lambda,u_m)-(1-\widetilde{\lambda})c(\gamma,\mu,u_m)\nonumber.
  \end{align}
  Using the fact that $\gamma(\mu_{ab}) \notin E_m$, we deduce, as in \eqref{eq:c positive Chapter cut locus}, that
	\begin{equation*}
		c(\gamma,\mu,u_m) \geq c(\gamma,\mu,d)>0.
	\end{equation*}
	This yields
	\begin{equation*}
		c(\widetilde{\gamma},\widetilde{\lambda},u_m)<c(\gamma,\lambda,u_m).
	\end{equation*}
	Moreover the unit speed geodesic $\widetilde{\gamma}:[\mu_{ab},b] \to \mathring{M}$ verifies $\widetilde{\gamma}((\mu_{ab},\widetilde{\lambda}_{\mu_{ab}b}]) \subset E_m$.
	Now, arguing as above, if there exists $t\in (\widetilde{\lambda},1)$ such that $\widetilde{\gamma}(t_{\mu_{ab}b})\notin E_m$, then we may build two numbers $\nu \in (\widetilde{\lambda},1)$ and $\widehat{\lambda} \in (0,1)$ such that the unit speed geodesic $\widehat{\gamma} := \widetilde{\gamma}_{|_{[\mu_{ab},\nu_{ab}]}}$ verifies
	\begin{equation*}
		c(\widehat{\gamma},\widehat{\lambda},u_m)< c(\widetilde{\gamma},\widetilde{\lambda},u_m),
	\end{equation*}
	and
	\begin{equation*}
		\widehat{\gamma}((\mu_{ab},\nu_{ab})) \subset E_m.
	\end{equation*}
	So we now need to show that
	\begin{equation}
    \label{eq:semiconcave bis} \inf\limits_{\gamma,\lambda} c(\gamma,\lambda,u_m)\geq 0,
  \end{equation}
	where the infimum is taken over unit speed geodesics $\gamma:[a,b] \to \mathring{M}$ such that $\gamma\big((a,b)\big) \subset E_m$.
  \medskip

	By continuity of $u_m$, \eqref{eq:semiconcave bis} is equivalent to simply saying that $u_m$ is $(C+\eta)$-semiconcave on $E_m$. Therefore, as $u_m$ is smooth on $E_m$, by \cref{prop:D2}, we only need to show the pointwise condition
  \begin{equation}
    \label{eq:local semiconcavity}
    D^2u_m\leq (C+\eta)Id \quad \text{on}\quad E_m.
  \end{equation}
	Now, let $C_1,C_2,C_3>0$ be some constants to be determined later, and $\varepsilon>0$ to be chosen small enough later. For $p\in E_m$ and $X\in \mathbb S^{n-1}(T_pM)$, we define
  \begin{equation}\label{eq:def function to maximise}
    f_\varepsilon(p,X) := D^2u_m(X,X) + \varepsilon\left(C_1 \abs{\nabla u_m}^2(p) + C_2 u_m^2(p) - C_3 u_m(p)\right).
  \end{equation}
  We will show that for a good choice of constants $C_1$, $C_2$, $C_3$, depending only of $M$ and $\abs{d}_{L^\infty}$, for any $\varepsilon>0$ small enough,  depending only on $M$, $\abs{d}_{L^\infty}$ and $\eta$, we have for any $m$ large enough,
  \begin{equation}\label{eq:bound on f}
    f_{\varepsilon}(p,X) \leq C+ \frac{2\eta}{3}\quad\text{for every}\quad p \in E_m\quad\text{and every}\quad X\in \mathbb S^{n-1}(T_pM).
  \end{equation}
  This will conclude the proof since, as $u_m$ is bounded by $\abs{d}_{L^\infty}$ and $1$-Lipschitz, we will then get
  \begin{equation*}
    D^2u_m(X,X) \leq C + \frac{2\eta}{3} + \varepsilon C(M,\abs{d}_{L^\infty}),
  \end{equation*}
	where $C(M,\abs{d}_{L^\infty})>0$ is a constant depending on $M$ and $\abs{d}_{L^\infty}$ only.
  But this implies \eqref{eq:local semiconcavity} if $\varepsilon$ has been taken small enough.
  \smallskip

  Suppose by contradiction that
  \begin{equation}\label{eq:maximality0}
 \sup_{\substack{p\in E_m \\ X\in \mathbb S^{n-1}(T_pM)}} f_\eps(p,X)>C+ \frac{2\eta}{3}.
  \end{equation}
  Let us assume that $m$ is large enough so that $D^2u_m\le (C+\sfrac{\eta}{3})Id$ in a neighborhood of $\partial E_m$. In particular, we get that for $\eps$ small enough, depending only on $M$, $\abs{d}_{L^\infty}$ and $\eta$,
  $$f_\eps< C+\frac{2\eta}{3}\quad\text{in a neighborhood of}\quad \partial E_m.$$
  Thus, by \eqref{eq:maximality0} and the precompactness of $E_m$, there exist $q \in E_m$ and $Y\in \mathbb S^{n-1}(T_{q}M)$ such that
  \begin{equation}\label{eq:maximality}
    f_\varepsilon(q,Y) = \sup_{\substack{p\in E_m \\ X\in \mathbb S^{n-1}(T_pM)}} f_\eps(p,X).
  \end{equation}
	In the following, $C^M$ will denote any constant that depends only on $M$.
  Let us pick some normal coordinates at $q$ such that $\partial_1 (q) = Y$. We then extend the vector $Y$ into a vector field (still denoted by $Y$) in a neighborhood of $q$, by setting $Y:= \partial_1/ \abs{\partial_1}$. As $D\partial_1 (q) = 0$, we also have $DY(q) = 0$.
  Moreover, as the manifold $M$ is compact, $D^2Y(q)$ is bounded by a constant that depends on $M$ only: we have
  \begin{equation}\label{eq:second derivative y bounded}
    \abs{D^2Y(q)}\leq C^M.
  \end{equation}
  We will show that the Laplacian of $p\mapsto f_\varepsilon(p,Y(p))$ is positive at $q$, which contradicts the maximality of $(q,Y(q))$ in \eqref{eq:maximality}. Let us estimate $\Delta (D^2u_m(Y,Y))$ at the point $q$, using the abstract index notation.
  \begin{align}
    \Delta (D^2u_m(Y,Y))
    & = g^{ab}D_aD_b(D^2_{cd}u_mY^cY^d) \nonumber\\
    & = g^{ab}\left(D^4_{abcd}u_mY^cY^d + D^3_{acd}u_mD_b(Y^cY^d)\right. \nonumber\\
    &\quad \quad \left.+ D^3_{bcd}u_mD_a(Y^cY^d) + D^2_{cd}u_mD^2_{ab}(Y^cY^d) \right). \label{eq:laplacian hessian estimate}
  \end{align}
  We may divide the right-hand side into four terms and estimate them at the point $q$ individually.
  The second term is null because it contains $D_b(Y^cY^d) = (D_bY^c)Y^d + Y^cD_bY^d$, and $DY = 0$. The third term is also null, for the same reason. By \eqref{eq:second derivative y bounded}, we can estimate the fourth term as follows:
  \begin{equation}\label{eq:fourth term}
    g^{ab}D^2_{cd}u_mD^2_{ab}(Y^cY^d)
    \geq -C^M\abs{D^2 u_m}\ge -\frac{C^M}{\varepsilon} - \varepsilon\abs{D^2 u_m}^2.
  \end{equation}
  It now remains to estimate the first term of \eqref{eq:laplacian hessian estimate}. Using the notation
  $$D_{[ab]}:=D_aD_b-D_bD_a,$$
we compute
  \begin{align*}
    D_aD_bD_cD_du_m &= D_aD_{[bc]}D_du_m + D_{[ac]}D_bD_du_m + D_cD_aD_{[bd]}u_m + D_cD_{[ad]}D_bu_m + D_cD_dD_aD_bu_m.
  \end{align*}
  By definition of the Riemann tensor we have
  \begin{equation*}
    D_aD_{[bc]}D_du_m=D_a(R_{bced}D^e u_m)=(D_aR_{bced})D^e u_m+R_{bced}D_aD^e u_m,
  \end{equation*}
  and so
  \begin{equation}\label{eq:term1}
    \abs{D_aD_{[bc]}D_du_m} \geq - C^M \abs{\nabla u_m} - C^M \abs{D^2 u_m}.
  \end{equation}
  Likewise,
  \begin{equation}\label{eq:term2}
    \abs{D_cD_{[ad]}D_bu_m}\geq - C^M \abs{\nabla u_m} - C^M \abs{D^2 u_m}.
  \end{equation}
  To compute the term $D_{[ac]}D_bD_du_m$, let us pick some coordinates $(x^i)$ and write $D_bD_du_m = D^2_{ij}u_m\mathrm{d}x^i_b\mathrm{d}x^j_d$. Then, we have
  \begin{align}
    D_{[ac]}D_bD_du_m
    &=D_{[ac]}(D^2_{ij}u_m\mathrm{d}x^i_b\mathrm{d}x^j_d) \nonumber \\
    &=(D_{[ac]}D^2_{ij}u_m)\mathrm{d}x^i_b\mathrm{d}x^j_d + D^2_{ij}u_m(D_{[ac]}\mathrm{d}x^i_b)\mathrm{d}x^j_d + D^2_{ij}u_m\mathrm{d}x^i_b(D_{[ac]}\mathrm{d}x^j_d) \nonumber \\
    &= 0 + D^2_{ij}u_mR_{aceb}(\mathrm{d}x^i)^e\mathrm{d}x^j_d + D^2_{ij}u_m \mathrm{d}x^i_b R_{aced}(\mathrm{d}x^j)^e \nonumber \\
    &= R_{aceb}D^eD_du_m+R_{aced}D_bD^eu_m,  \nonumber
  \end{align}
  and so
  \begin{equation}\label{eq:term3}
    \abs{D_{[ac]}D_bD_du_m}\geq - C^M \abs{D^2 u_m}.
  \end{equation}
  By symmetry of the tensor $D^2u_m$, we have
  \begin{equation}
    \label{eq:term4} D_cD_aD_{[bd]}u_m = 0.
  \end{equation}
  Putting \eqref{eq:term1}, \eqref{eq:term2}, \eqref{eq:term3} and \eqref{eq:term4} together, we find
  \begin{equation*}
    \abs{g^{ab}D_aD_bD_cD_du_m} \geq - C^M \abs{\nabla u_m} - C^M \abs{D^2 u_m} - \abs{g^{ab}D_cD_dD_aD_bu_m}.
  \end{equation*}
  In $E_m$, $u_m$ has constant Laplacian, so
  \begin{equation*}
    g^{ab}D_cD_dD_aD_bu_m = D_cD_dg^{ab}D_aD_bu_m =
    D_cD_d \Delta u_m = 0.
  \end{equation*}
  So we get
  \begin{equation*}
    \abs{g^{ab}D_aD_bD_cD_du_m} \geq - C^M \abs{\nabla u_m} - C^M \abs{D^2 u_m}.
  \end{equation*}
  From this and the fact $Y$ has norm $1$, we deduce
  \begin{equation*}
    \abs{g^{ab}D_aD_bD_cD_du_mY^cY^d} \geq - C^M \varepsilon^{-1} - \varepsilon\abs{\nabla u_m}^2 - \varepsilon\abs{D^2 u_m}^2.
  \end{equation*}
  Combining this equation with \eqref{eq:laplacian hessian estimate} and \eqref{eq:fourth term}, we obtain at the point $q$,
  \begin{equation}\label{eq:laplacian of maximal second derivative}
    \Delta (D^2u_m(Y,Y)) \geq - C^M \varepsilon^{-1} -2 \varepsilon \abs{D^2u_m}^2 - \varepsilon \abs{\nabla u_m}^2.
  \end{equation}
  We recall the Bochner-Weitzenböck formula:
  \begin{equation*}
    \Delta \big(\abs{\nabla u_m}^2\big) = 2\mathrm{Ric}(\nabla u_m,\nabla u_m)+2\abs{D^2u_m}^2+2(\nabla \Delta u_m,\nabla u_m).
  \end{equation*}
  As $M$ is compact, there exists a constant $K>0$ such that $\mathrm{Ric} \geq -K$. Using the fact that $u_m$ has constant Laplacian in $E_m$, we get
  \begin{equation}\label{eq:laplacian norm gradient square estimate}
    \Delta \big(\abs{\nabla u_m}^2 \big)\geq 2\abs{D^2u_m}^2 - 2K \abs{\nabla u_m}^2.
  \end{equation}
  Furthermore, using the fact that $\Delta u_m = -2m$ in $E_m$ again, we find
  \begin{align}
    \Delta \big(u_m^2\big) & = 2\abs{\nabla u_m}^2 - 2 m u_m \nonumber \\
		&\geq 2\abs{\nabla u_m}^2 - 2 m \abs{d}_{L^\infty},\label{eq:laplacian u square estimate} \\
    \Delta u_m & = - 2 m. \label{eq:laplacian u}
  \end{align}
  Using \eqref{eq:laplacian norm gradient square estimate}, \eqref{eq:laplacian u square estimate} and \eqref{eq:laplacian u}, we get
  \begin{equation*}
    \Delta\left( \varepsilon\abs{\nabla u_m}^2  + (K+1)\varepsilon u_m^2-((K+1)\abs{d}_{L^\infty} +1 )\varepsilon u_m\right)\geq 2 \varepsilon \abs{D^2u_m}^2 + \varepsilon \abs{\nabla u_m}^2 + \varepsilon m.
  \end{equation*}

   Setting $(C_1,C_2,C_3) = (1, K+1, ((K+1)\abs{d}_{L^\infty} +1 ))$, and recalling the definition of $f_\varepsilon$ \eqref{eq:def function to maximise}, we obtain thanks to \eqref{eq:laplacian of maximal second derivative}:
  \begin{equation*}
    \Delta (f_\varepsilon(p,Y(p)))_{p = q} \geq -C^M \varepsilon^{-1} + \varepsilon m.
  \end{equation*}
  In particular, if $m$ is large enough, depending on $M$ and $\varepsilon$, this contradicts the maximality of $(q,Y(q))$ in \eqref{eq:maximality}.
  This concludes the proof of \eqref{eq:bound on f} and \cref{prop:semiconcavity}.\
  \end{proof}

In order to apply \cref{prop:semiconcavity} to problem \eqref{eq:distance constraint}, we will need the following lemma.
\begin{lemma}[Bound of $D^2u_m$ near $\partial E_m$]\label{p:bound-on-D2-near-the-boundary}
Let $u_m$ be the solution of \eqref{eq:gradient constraint manifold}, as in \cref{t:main-main}. Let $\varepsilon>0$ be smaller than the distance from $b$ to ${\rm Cut}_b(M)$.
Let $E_m := \{u_m<d_b\}$. From Proposition \ref{thm:non-contact set converges to cut locus}, we know that for $m$ large enough, we have $\overline{E_m}\subset M \setminus B(b,\varepsilon)$.
Let $C>0$ be such that $d_b$ is $C$-semiconcave on $M\setminus B(b,\varepsilon)$.
Then, for any $m$ large enough, for any $\eta >0$, there is a neighborhood $\mathcal N_{\eta,m}$ of $\partial  E_m$ in $M\setminus B(b,\varepsilon)$ such that
\begin{equation}
	\label{eq:boundary values second order derivative}
	D^2u_m\le \left(C+\eta \right)Id\quad\text{in}\quad E_m\cap \mathcal N_{\eta,m}.
\end{equation}
\end{lemma}

\begin{proof}
  We will use a theorem for obstacle problems on $\rn$. Let us show that $u_m$ is the solution of an obstacle problem on an open subset of $\R^n$. Then, we will apply \cite[Chapter 2, Theorem 3.8]{friedman1982variational} to conclude that \eqref{eq:boundary values second order derivative} holds.

  The minimality of $u_m$ in \eqref{eq:distance constraint} implies
  \begin{equation}
    -\Delta u_m - 2m \geq 0, \quad u_m \leq d_b \quad \text{and} \quad (-\Delta u_m - 2m)(u_m-d_b) = 0. \label{eq:obstacle problem}
  \end{equation}
  Let $\widetilde{\Omega}$ be defined as in the proof of \cref{prop:regularity}. Let $\phi:\widetilde{\Omega}\to\widetilde{U}$ be a normal coordinates chart. Writing down \eqref{eq:obstacle problem} in these coordinates, we find
  \begin{equation*}
    A\widetilde{u_m}-2m \geq 0, \quad \widetilde{u_m} \leq \psi \quad \text{and}\quad (A\widetilde{u_m}-2m)(\widetilde{u_m}-\psi) = 0,
  \end{equation*}
  where $A$ is the Laplacian of $M$ in the coordinates defined by $\phi$, $\widetilde{u_m} = u_m \circ \phi^{-1}$ and $\psi = d_b \circ \phi^{-1}$.
  This is the form of \cite[Chapter 2, equation (3.16)]{friedman1982variational}, so we can apply \cite[Chapter 2, Theorem 3.8]{friedman1982variational}, to deduce that
  \begin{equation}\label{eq:boundary condition second derivatives}
    \forall p \in \partial E_m, \forall X\in\R^n \lim_{\substack{q \to p \\ q \in E_m}} D^2 \widetilde{u_m}(\phi(q))(X,X) \leq D^2\psi(\phi(p))(X,X).
  \end{equation}
  Moreover, we have
  \begin{align*}
    D^2 \widetilde{u_m} &= D^2 u_m \circ (D\phi^{-1},D\phi^{-1}) + Du_m \circ D^2\phi^{-1},\\
    D^2 \psi &= D^2 d_b \circ (D\phi^{-1},D\phi^{-1}) + Dd_b \circ D^2\phi^{-1},
  \end{align*}
  and $D u_m = D d_b$ on $\partial E_m$ because $u_m$ is $C^1$.
  Thus, \eqref{eq:boundary condition second derivatives} yields:
  \begin{equation*}
    \forall p \in \partial E_m, \forall X\in\R^n \lim_{\substack{q \to p \\ q \in E_m}} D^2 u_m(q)(X_q,X_q) \leq D^2 d_b(p)(X_p,X_p),
  \end{equation*}
  where we have set $X_q := D\phi^{-1}(\phi(q))X$. As $d_b$ is $C$-semiconcave, with \cref{prop:D2}, we get
  \begin{equation}\label{eq:limsup second derivative}
    \forall p \in \partial E_m, \forall X\in\R^n \lim_{\substack{q \to p \\ q \in E_m}} D^2 u_m(q)(X_q,X_q) \leq C \abs{X_p}^2.
  \end{equation}
  From there, we deduce that
  \begin{equation}\label{eq:boundary condition}
    \text{for } q\in E_m \text{ close enough to } \partial E_m, \text{ we have }D^2 u_m(q) \leq C+\eta.
  \end{equation}
  Indeed, if not, there exist a sequence $(q_k)$ of points of $E_m$ whose distance to $\partial E_m$ goes to $0$, and a sequence $(X_k)$ of unit vectors of $\R^2$ such that for any $k\in \N$,
  \begin{equation}\label{eq:second derivative contradiction}
    D^2 u_m(q_k)\big((X_k)_{q_k},(X_k)_{q_k}\big) > C + \eta.
  \end{equation}
  As $E_m$ is precompact, up to extracting a subsequence, we can assume that $(q_k)$ converges to a point $p \in \partial E_m$, and $(X_k)$ converges to a vector $Y \in \R^n$. Because of \eqref{eq:limsup second derivative}, we have
  \begin{equation}\label{eq:limsup second derivative bis}
    \lim_{k\to \infty}D^2 u_m(q_k)(Y_{q_k},Y_{q_k})\leq C.
  \end{equation}
  Furthermore, we know from \cref{prop:regularity} that $D^2u_m$ is locally bounded. As $(X_k)_{q_k} - Y_{q_k}$ converges to $0$ when $k$ goes to $\infty$, this implies
  \begin{equation}\label{eq:second derivative difference}
    \lim_{k\to \infty}D^2 u_m(q_k)\big((X_k)_{q_k},(X_k)_{q_k}\big) - D^2 u_m(q_k)(Y_{q_k},Y_{q_k}) = 0.
  \end{equation}
  Inequalities \eqref{eq:second derivative contradiction}, \eqref{eq:limsup second derivative bis} and \eqref{eq:second derivative difference} yield a contradiction.
  So \eqref{eq:boundary condition} is true. This concludes the proof.
\end{proof}

\section{Convergence of the gradients}\label{s:gradients}
In this section, we show that the uniform semiconcavity of $u_m$ implies the convergence of the gradients in the sense of \cref{t:main-main} \ref{item:gradients}. We notice that the results from this section also apply to more general sequences of semiconcave functions.
\subsection{Lower semicontinuity}\label{sub:lsc}
In this section, we prove the first inequality in \cref{t:main-main} \ref{item:gradients} (see \cref{prop:gradient lsc}). We start by the following lemma.
\begin{lemma}\label{l:under tangent plane}
	Let $u:M\to\R$ be a $C$-semiconcave function. Let $p,q\in M$ be such that there exists a geodesic from $p$ to $q$. Then,
	\begin{equation*}
	u(q) \leq u(p) + \abs{\nabla u(p)}d(p,q) + \frac{C}2 d(p,q)^2,
	\end{equation*}
	where $|\nabla u|(p)$ is the norm of the generalized gradient, defined in \eqref{e:def-gradient}.
\end{lemma}
\begin{proof}
	Let $\gamma:[0,d(p,q)]\to M$ be a geodesic from $p$ to $q$. Consider the function $f(t)=\frac12 Ct^2 - u(\gamma(t))$. By the semiconcavity of $u$, we know that $f$ is convex. Thus, we have
	\begin{equation*}
	f(d(p,q)) \geq f(0) + f'(0)d(p,q).
	\end{equation*}
	On the other hand, setting $\dot\gamma(0):=v\in T_p(M)$, by construction, we have
	$$f(0)=-u(p)\,,\quad f(d(p,q))=\frac{C}{2}d(p,q)^2-u(q)\,,\quad\text{and}\quad f'(0)=-\partial_v^+ u(p).$$
Thus, we obtain

\qquad\qquad\qquad $\displaystyle u(q) \leq u(p) + d(p,q)\partial_v^+ u(p) + \frac{C}2d(p,q)^2\le  u(p) + \abs{\nabla u(p)}d(p,q) + \frac{C}2 d(p,q)^2.$
\end{proof}

\begin{proposition}\label{prop:gradient lsc}
Let $M$ be a Riemannian manifold and let $C>0$ be a fixed constant. Let $u_k:M\to\R$ be a sequence of $C$-semiconcave continuous functions that converges locally uniformly to a continuous function $u_\infty:M\to\R$. Then,  $u_\infty$ is also $C$-semiconcave, and for any sequence of points $p_k\to p_\infty\in M$, we have
	\begin{equation}\label{e:lsc-gradient-general}
 \abs{\nabla u_\infty}(p_\infty)\le 	\liminf_{k \to \infty}\abs{\nabla u_{k}}(p_k).
	\end{equation}
\end{proposition}
\begin{proof}
First, notice that the $C$-semiconcavity of $u_\infty$ is an immediate consequence of the pointwise convergence and the $C$-semiconcavty of $u_k$. In particular, the generalized gradients $\abs{\nabla u_{m_k}}(p_k)$ and $\abs{\nabla u_\infty}(p_\infty)$ are well-defined by \cref{prop:generalized gradient}. Thus, we only need to prove \eqref{e:lsc-gradient-general}. We notice that \eqref{e:lsc-gradient-general} is trivial if $\abs{\nabla u_\infty}(p_\infty)=0$. Thus, we suppose that $\abs{\nabla u_\infty}(p_\infty)>0$. In particular, there are a vector $v\in \mathbb S^{n-1}(T_{p_\infty}M)$ and a unit speed geodesic $\gamma$ with $\gamma (0)=p_\infty$ and $\dot\gamma(0)=v$ such that
$$\abs{\nabla u_\infty(p_\infty)}=\lim_{t\to0^+}\frac{u_\infty(\gamma(t))-u_\infty(p_\infty)}{t}.$$
In particular, for any $\varepsilon>0$, we can find $q\in M$ such that $d(p_\infty,q)\le \eps$ and
	\begin{equation*}
	\abs{\nabla u_\infty(p_\infty)} \leq \frac{u_\infty(q) - u_\infty(p_\infty)}{d(q,p_\infty)} + \varepsilon.
	\end{equation*}
	Then, by the uniform convergence of $u_k$ and \cref{l:under tangent plane}, we get
	\begin{align*}
	\abs{\nabla u_\infty(p_\infty)}
	 \leq \liminf_{k\to\infty}\frac{u_k(q) - u_k(p_k)}{d(q,p_k)} + \varepsilon  &\leq \liminf_{k\to\infty} \abs{\nabla u_k(p_k)} + \frac{C}2d(q,p_k) +\varepsilon \\
	 &\le \liminf_{k\to\infty} \abs{\nabla u_k(p_k)} + (C+1)\varepsilon,
	\end{align*}
	which concludes the proof, as the inequality holds for any $\varepsilon$.
\end{proof}

\subsection{Proof of \cref{t:main-main} \ref{item:gradients}}\label{sub:grad-continuity}
The claim \eqref{e:inequality} follows from \cref{prop:gradient lsc}. Thus, we only need to prove \eqref{e:equality}. First, notice that, if $|\nabla d_b|(p_\infty)=1$, then \eqref{e:equality} follows from  \eqref{e:inequality} and the fact that $u_m$ is $1$-Lipschitz. Let now $|\nabla d_b|(p_\infty)<1$.  Suppose by contradiction that there are a subsequence $m_k \tend{k}{+\infty}+\infty$ and constants $\eps>0$ and $\eta_0>0$ such that
$$|\nabla d_b|(p_\infty)+\eps\le |\nabla u_{m_k}|(p)\quad\text{for every}\quad p\in B_{\eta_0}(p_\infty)\quad\text{and every}\quad k\ge 0.$$
We now fix $\eta\le\eta_0$, which will be chosen later in the proof.
	Let $(q_t)_{t\geq0}$ be the curve defined by
	\begin{equation*}
	q_0 = p_\infty \quad \text{and} \quad \frac{\mathrm{d}q_t}{\mathrm{d}t} = \nabla u_{m_k} (q_t).
	\end{equation*}
	Let $T>0$ be such that for any $t\in[0,T]$, $d(q_t,p_\infty)\leq\eta$, and in particular
	$$|\nabla d_b|(p_\infty)+\eps\le |\nabla u_{m_k}|(q_t)\quad\text{for every}\quad t\in[0,T].$$
	We have
	\begin{align*}
	u_{m_k}(q_T)-u_{m_k}(p_\infty) & =\int_0^T \abs{\nabla u_{m_k}(q_t)}^2 \mathrm{d}t \geq \int_0^T \big(|\nabla d_b|(p_\infty)+\eps\big)^2\,\mathrm{d}t =  T  \big(|\nabla d_b|(p_\infty)+\eps\big)^2.
	\end{align*}
	As $u_{m_k}$ is bounded by the diameter of $M$, this estimate implies that there exists a finite biggest time $T>0$ such that for any $t\in[0,T]$, $d(q_t,p_\infty)\leq\eta$. In particular, $d(p_\infty,q_T) = \eta$.
	Let $\gamma$ be a unit speed minimizing geodesic between $p_\infty$ and $q_T$. By \cref{prop:db semiconcave}, there is a constant $C_d>0$ such that $d_b$ is $C_d$-semiconcave in $B_{\eta_0}(p_\infty)$. In particular, by \cref{l:under tangent plane},  we have that
	\begin{align}
	d_b(q_T)-d_b(p_\infty)
	&\leq \abs{\nabla d_b(p_\infty)}d(p_\infty,q_T) +  C_d(d(p_\infty,q_T))^2= \abs{\nabla d_b(p_\infty)}\eta +  C_d \eta^2\ .\label{eq:dvariation}
	\end{align}
	On the other hand,
	\begin{align}
	u_{m_k}(q_T)-u_{m_k}(p_\infty) & =\int_{0}^T \abs{\nabla u_{m_k}(q_t)} \abs{\frac{\mathrm{d}q_t}{\mathrm{d}t}}\mathrm{d}t\geq \int_0^T \big(|\nabla d_b|(p_\infty)+\eps\big)\abs{\frac{\mathrm{d}q_t}{\mathrm{d}t}}   \mathrm{d}t \nonumber\\
	&= \big(|\nabla d_b|(p_\infty)+\eps\big)\int_0^T \abs{\frac{\mathrm{d}q_t}{\mathrm{d}t}}   \mathrm{d}t \geq \big(|\nabla d_b|(p_\infty)+\eps\big)d(q_0,q_T)\nonumber\\
	&=  \big(|\nabla d_b|(p_\infty)+\eps\big)\eta. \label{eq:uvariation}
	\end{align}
	Combining \eqref{eq:dvariation} and \eqref{eq:uvariation}, we get that
	$$ \eps\eta-C_d \eta^2\le \Big(u_{m_k}(q_T)-u_{m_k}(p_\infty)\Big)-\Big(d_b(q_T)-d_b(p_\infty)\Big)\le 2\|u_{m_k}-d_b\|_{L^\infty(M)}.$$
	Now, taking $\eta$ small enough, we get that
	$$ \frac12\eps\eta\le 2\|u_{m_k}-d_b\|_{L^\infty(M)}\quad\text{for every}\quad k\ge 0,$$
	but this is in contradiction with the uniform convergence of $u_m$ to $d_b$.
\qed


\appendix
\section{Appendix about semiconcavity}\label{app:semiconcavity}
In this section we prove that defining local semiconcavity through charts (as in \cite{mantegazza_mennucci_2003}), or through geodesics, is the same (see \cref{prop:definitions semiconcavity_intro}).
We recall the notation $\lambda_{ab} = (1-\lambda)a + \lambda b$, for $a,b \in \R$ and $\lambda \in [0,1]$ and we notice that the $C$-semiconcavity of $u:M\to\R$ (in the sense of \cref{d:semiconcavity}) can be rewritten as
  \begin{equation*}
    \lambda_{u(\gamma(a))u(\gamma(b))} - u(\gamma(\lambda_{ab})) \leq C \lambda (1-\lambda) (b-a)^2,
  \end{equation*}
for every unit speed geodesic $\gamma:[a,b] \to M$ and any $\lambda \in [0,1]$.

In order to prove \cref{prop:definitions semiconcavity_intro}, we need the following lemma, which shows how to estimate the difference between two geodesics linking a pair of given points, for two different metrics.
\begin{lemma}\label{lemma:geodesic estimate}
  Let $g$ be a metric on the unit ball $B_1(0) \subset \rn$. There exists a constant $B>0$ such that for any unit speed geodesic $\gamma : [a,b] \to (B_1(0),g)$ and $\lambda\in[0,1]$, we have
  \begin{equation*}
    \abs{\gamma(\lambda_{ab}) - \lambda_{\gamma(a)\gamma(b)}} \leq B \lambda(1-\lambda)(b-a)^2.
  \end{equation*}
\end{lemma}
\begin{proof}
  It suffices to prove that the estimate holds for $\lambda \leq \frac{1}{2}$, as the case $\lambda \geq \frac{1}{2}$ can be deduced by considering $\widetilde{\gamma} : t \mapsto \gamma(b-t)$ instead of $\gamma$.
  A unit speed geodesic  $\gamma: [a,b] \to (B_1(0),g)$ satisfies the geodesic equation
  \begin{equation*}
    \ddot{\gamma}^l + \Gamma^l_{ij}\dot{\gamma}^i\dot{\gamma}^j = 0,
  \end{equation*}
  where $\Gamma^l_{ij}$ are the Christoffel symbols of the metric $g$. As $\gamma$ is unit speed, the $(\dot{\gamma}^i)$ are bounded, uniformly in $\gamma$. Therefore, there exists a constant $\alpha>0$ independent of $\gamma$ such that $\abs{\ddot{\gamma}} \leq \alpha$. By integration, we find $$\abs{\gamma(t) - \gamma(a) - \dot{\gamma}(a)(t-a)} \leq \alpha(t-a)^2.$$
  Evaluating this expression at $b$ yields $$\abs{\gamma(b) - \gamma(a) - \dot{\gamma}(a)(b-a)} \leq \alpha (b-a)^2.$$ From these two estimates, we deduce
  \begin{equation*}
    \abs{\gamma(t) - \gamma(a) - \frac{\gamma(b) - \gamma(a)}{b-a}(t-a)} \leq \alpha (t-a)^2 + \alpha (b-a)(t-a).
  \end{equation*}
  Taking $t = (1-\lambda)a +\lambda b$ in this estimate yields
  \begin{align*}
    \abs{\gamma((1-\lambda)a + \lambda b) -\left((1-\lambda)\gamma(a) + \lambda \gamma(b)\right)}
    & \leq \alpha \lambda(1+\lambda)(b-a)^2  =\frac{ \alpha(1+\lambda)}{1-\lambda}\lambda(1-\lambda)(b-a)^2.
  \end{align*}
  Taking $B := \frac{ \alpha(1+\sfrac12)}{1-\sfrac12}$, this proves the desired estimate when $\lambda \leq \sfrac{1}{2}$. This concludes the proof.
\end{proof}

\begin{proof}[Proof of \cref{prop:definitions semiconcavity_intro}]
  Let us assume that $u$ is locally semiconcave. Let $\psi:U\to V$ be a chart from an open set $U$ of $M$ to on open set $V$ of $\rn$, and $y \in V$. Let $f:= u \circ \psi^{-1}$. We want to show that $f$ is semiconcave in a neighborhood of $y$, as a function of $\rn$. We first observe that $f$ is locally semiconcave on the manifold $(V,\psi_{\star}g)$.
  Let $V'\subset V$ be a neighborhood of $y$ that is geodesically convex for the metric $\psi_{\star}g$, and such that there exists a constant $C>0$ such that $f$ is $C$-semiconcave on $(V',\psi_{\star}g)$. Let $d$ denote the distance function on $(V',\psi_{\star}g)$. Up to taking $V'$ smaller, we may assume that the metric $\psi_{\star}g$ is bounded on $V'$, and so there exists a constant $\beta>0$ such that
  \begin{equation*}
    \forall x,y\in V', \quad d(x,y)\leq \beta \abs{x-y}.
  \end{equation*}
  Let $x,y\in V'$ be such that $[x,y] \subset V'$, and $\lambda \in [0,1]$.
  Let $\gamma : [a,b] \to V'$ be a unit speed geodesic of $(V',\psi_{\star}g)$ from $x$ to $y$. By the $C$-semiconcavity of $f$ on $(V',\psi_{\star}g)$, we have
  \begin{align*}
    \lambda_{f(x)f(y)} - f(\lambda_{xy})
    & = \lambda_{f(\gamma(a))f(\gamma(b))} - f(\lambda_{\gamma(a)\gamma(b)}) \\
    & \leq C\lambda(1-\lambda)(b-a)^2 + f(\gamma(\lambda_{ab})) - f(\lambda_{\gamma(a)\gamma(b)}) \\
    & \leq C\lambda(1-\lambda)(b-a)^2 + \mathrm{Lip}(f)\abs{\gamma(\lambda_{ab}) - \lambda_{\gamma(a)\gamma(b)}}.
  \end{align*}
  Applying \cref{lemma:geodesic estimate} above, we get a constant $B>0$ such that
  \begin{align*}
    \lambda_{f(x)f(y)} - f(\lambda_{xy})
    & \leq (C+ \mathrm{Lip}(f)B)\lambda(1-\lambda)(b-a)^2 \\
    & = (C+ \mathrm{Lip}(f)B)\lambda(1-\lambda)(d(x,y))^2 \\
    & \leq (C+ \mathrm{Lip}(f)B)\beta^2 \lambda(1-\lambda)\abs{x-y}^2,
  \end{align*}
  and so $f$ is semiconcave on $V'$, as a function of $\rn$.

  Reciprocally, let us assume that $u\circ \psi^{-1}$ is locally semiconcave as a function of $\rn$ for any chart $\psi$. Then, we can show that $u\circ \psi^{-1}$ is locally semiconcave for the metric $\psi_{\star}g$, for any chart $\psi$, by using the same technique. From there we deduce that $u$ is locally semiconcave. This concludes the proof.
\end{proof}

\section{A counter-example to the equivalence of \eqref{eq:gradient constraint manifold} and \eqref{eq:distance constraint} for small $m$}\label{s:esempio}

\begin{theorem}\label{prop:counterexample}
	There exist a surface of revolution $M$ and a parameter $m>0$ such that $u_m \neq u_m^d$.
\end{theorem}

\begin{proof}[Proof of \cref{prop:counterexample}]
	Let $r_\theta$ denote the rotation of $\R^3$ of angle $\theta\in[0,2\pi)$ around the $z$-axis.\\ Let $T:=10^{10}$ and $r,h:[0,T] \to \R$ be two smooth functions such that:

\noindent\begin{minipage}{0.66\textwidth}	\begin{align}
	\gamma &: t \mapsto (r(t),0,h(t)) \; \text{is a unit speed curve}. \nonumber \\
	M &:= \{ r_{\theta}(\gamma(t)):(t,\theta)\in[0,T]\times[0,2\pi)]\} \; \text{is a smooth surface,}\nonumber \\
	r(0)&=r(T)=0, \nonumber \\
	r &\leq 2, \nonumber \\
	r([1,2]) &\subset [1,2], \nonumber \\
	r([3,4]) &\subset (0,10^{-10}), \nonumber \\
	r([5,T-1]) &\subset [1,2]. \nonumber
	\end{align}
		This information is pictured in Figure \ref{fig:gamma}. We chose $b = (0,0,0)$ as the base point on $M$, and $m=10^{-10}$. Let us assume that $u_m^d = u_m$ and build a better competitor in \eqref{eq:distance constraint} to contradict the minimality of $u_m^d$. We will first reduce \eqref{eq:distance constraint} to a one-dimensional problem. Note that the functional we are minimizing is rotation-invariant. More precisely, for any $\theta\in (0,2\pi)$ and $u\in H^1(M)$, we have
		\begin{equation}
		\int_{M} \abs{\nabla (u\circ r_\theta)}^2 -m(u\circ r_\theta) =
		\int_{M} \abs{\nabla u}^2 -mu.
		\end{equation}
			By the uniqueness of the minimizer $u_m^d$, we deduce that $u_m^d$ is rotation-invariant, \text{i.e.} there exists a function $\rho_m:[0,T]\to\R$ such that for any $\theta \in[0,2\pi)$ and $t\in[0,T]$, $u_m^d(r_{\theta}(\gamma(t))) = \rho_m(t)$.  Thus $u_m^d$ is a minimizer of \eqref{eq:distance constraint} among rotation-invariant functions. \vspace{0.1cm}
\end{minipage}
\begin{minipage}{0.32\textwidth}
		\begin{center}
			\begin{tikzpicture}
			\draw[->] (-1,0) -- (2,0);
			\draw (2,0) node[below right] {$x$};
			\draw [->] (0,-1) -- (0,6);
			\draw (0,6) node[above left] {$z$};
			\draw [domain=-pi/2:pi/2] plot ({1+(1/pi)*cos(\x r)},{1/pi+(1/pi)*sin(\x r)});
			\draw (1,2/pi)--(0.2,2/pi);
			\draw [domain=-pi:-pi/2] plot ({0.2+0.1*cos(\x r)},{2/pi+0.1+0.1*sin(\x r))});
			\draw (0.1,2/pi+0.1)--(0.1,2/pi+0.1+1);
			\draw [domain=pi/2:pi] plot ({0.2+0.1*cos(\x r)},{2/pi+0.1+1+0.1*sin(\x r)});
			\draw (0.2,2/pi+0.1+1+0.1)--(1,2/pi+0.1+1+0.1);
			\draw [domain=-pi/2:0] plot ({1+(1/pi)*cos(\x r)},{2/pi+0.1+1+0.1+1/pi+(1/pi)*sin(\x r)});
			\draw (1+1/pi,2/pi+0.1+1+0.1+1/pi)--(1+1/pi,5);
			\draw [domain=0:pi/2] plot ({1+(1/pi)*cos(\x r)},{5+(1/pi)*sin(\x r)});
			\draw (1,5+1/pi)--(0,5+1/pi);
			\draw (0,0) node[below left] {$\gamma(0)$}; \draw (0,0) node {$\bullet$};
			\draw (1,0) node[below] {$\gamma(1)$}; \draw (1,0) node {$\bullet$};
			\draw (1,0) node[below] {$\gamma(1)$}; \draw (1,0) node {$\bullet$};
			\draw (1,2/pi) node[above right] {$\gamma(2)$}; \draw (1,2/pi) node {$\bullet$};
			\draw (0.1,2/pi+0.1) node[left] {$\gamma(3)$}; \draw (0.1,2/pi+0.1) node {$\bullet$};
			\draw (0.1,2/pi+0.1+1) node[left] {$\gamma(4)$}; \draw (0.1,2/pi+0.1+1) node {$\bullet$};
			\draw (1,2/pi+0.1+1+0.1) node[below right] {$\gamma(5)$}; \draw (1,2/pi+0.1+1+0.1) node {$\bullet$};
			\draw (1,5+1/pi) node[above right] {$\gamma(T-1)$}; \draw (1,5+1/pi) node {$\bullet$};
			\draw (0,5+1/pi) node[left] {$\gamma(T)$}; \draw (0,5+1/pi) node {$\bullet$};
			\draw[->] (0.1,3)--(1.1,3);\draw (0.6,3) node[above] {$r(t)$};
			\draw (1+1/pi,3) node {$\bullet$}; \draw (1+1/pi,3) node[right] {$\gamma(t)$};

			\end{tikzpicture}

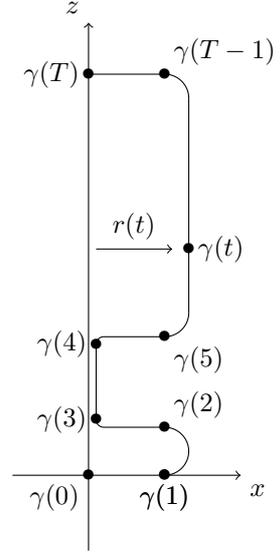
				{\sc Figure 3. \it The curve $\gamma$.}
				\vspace{0.2cm}
		\end{center}

	\label{fig:gamma}
\end{minipage}

Let $u:M\to \R$ be any rotation-invariant function, and $\rho:[0,T]\to\R$ be such that for any $\theta \in [0,2\pi)$, $u(r_{\theta}(\gamma(t))) = \rho(t)$.
	We will translate the minimization problem \eqref{eq:distance constraint} on $u$ into a problem on $\rho$.

	First, because $M$ is a surface of revolution, all the geodesics starting from $b = (0,0,0)$ have a constant angle $\theta$. Thus, they are of the form $t\mapsto r_\theta(\gamma(t))$ for some $\theta \in [0,2\pi)$. These are actually unit speed geodesics as $\gamma$ is unit speed. Hence,
	$d_b(r_\theta(\gamma(t)))=t$, and the constraint $u\leq d_b$ in \eqref{eq:distance constraint} is equivalent to $\rho(t)\leq t$.

	Secondly, we translate the $H^1$ constraint. To this end, let us define some coordinates $(t,\theta)$ on $M$ via the map
	$$\phi :(0,T)\times(0,2\pi)\to M\ ,\quad \phi(t,\theta)=r_\theta(\gamma(t)).$$
	We have
	\begin{align}
	\int_{M} \abs{\nabla u}^2
	&= \int_0^{2\pi}\int_0^T (\abs{\nabla u}^2\circ \phi ) J\phi \,\mathrm{d}t\,\mathrm{d}\theta \nonumber \\
	&= \int_0^{2\pi}\int_0^T \abs{\nabla u}^2(r_\theta(\gamma(t)))  r(t) \,\mathrm{d}t\,\mathrm{d}\theta = 2\pi\int_0^T\abs{\nabla u}^2(\gamma(t)) r(t) \,\mathrm{d}t, \label{eq:holder norm in coordinates}
	\end{align}
	because $u$ is rotation-invariant. Moreover, as $u$ is rotation-invariant, its gradient at the point $\gamma(t)$ is parallel to $\gamma'(t)$, and so
	\[\abs{\rho'(t)} = \abs{\nabla u(\gamma(t)) \cdot \gamma'(t)} = \abs{\nabla u (\gamma(t))} \abs{\gamma'(t)} = \abs{\nabla u (\gamma(t))}. \]
	Hence \eqref{eq:holder norm in coordinates} gives
	\begin{equation}
	\nonumber
	\int_{M} \abs{\nabla u}^2 = 2\pi\int_0^T\rho'(t)^2  r(t)\,dt
	\end{equation}
	Thus, the constraint $u\in H^1(M)$ in \eqref{eq:distance constraint} is equivalent to $v\in H^1((0,T),r(t)\mathrm{d}t)$.

	Thirdly, we may compute the functional likewise:
	\begin{equation}
	\nonumber
	\int_{M} \abs{\nabla u}^2 -mu = 2\pi\int_0^T\left(\rho'(t)^2 -m\rho(t) \right) r(t) \mathrm{d}t.
	\end{equation}
	All in all, as $u_m^d$ is a minimizer in \eqref{eq:distance constraint}, $\rho_m$ is a minimizer of :
	\begin{equation}
	\label{eq:one dimensional problem}
	\inf\left\{ \int_0^T\left(\rho'(t)^2 -m\rho(t) \right) r(t) \mathrm{d}t\ :\ \rho\in H^1\big((0,T),r(t)\mathrm{d}t\big),\ \rho(t)\leq t\right\}.
	\end{equation}
	The idea of the rest of the proof is the following. First, we recall the assumption $u_m^d = u_m$, which means that $\abs{\nabla u_m^d}\leq 1$, and so $\abs{\rho_m'}\leq1$. Now, if $\rho_m(4)$ is close to $4$, then $\rho_m'(t)$ is close to $1$ for $t\leq 4$, so a competitor $v$ such that $\rho'(t)$ is small for $t\leq 4$ will contradict the minimality of $\rho_m$ in \eqref{eq:one dimensional problem}.
	If on the contrary $\rho_m(4)$ is significantly smaller than $4$, then for $t\geq4$, $\rho_m(t)$ will be significantly smaller than $t$, so a competitor $\rho$ such that $\rho(t)$ is closer to $t$ for $t\geq4$ will contradict the minimality of $\rho_m$ in \eqref{eq:one dimensional problem}.
	Because we chose $r$ very small on the interval $[3,4]$ (see Figure 3), we can define a competitor $\rho$ independently on $[0,3]$ and $[4,T]$, without paying much for the behavior of $\rho$ on $[3,4]$.

	\emph{Case one: $\rho_m(4)\in[3.5,4]$.} Let us define a competitor $\rho$ for \eqref{eq:one dimensional problem}:
	\begin{equation}\nonumber
\rho:[0,T]\to \R\ ,\quad \rho(t)=
		\begin{cases}
		0 & \text{if} \quad  t\in[0,3] \\
		4(t-3)  & \text{if} \quad t\in[3,4]  \\
		\rho_m(t)+4-\rho_m(4)  & \text{if} \quad t \geq 4
		\end{cases}.
	\end{equation}
	Let us call $\mathcal{F}(\rho)$ the functional appearing in $\eqref{eq:one dimensional problem}$. We have, from the definition of $r$ and $\rho$,
	\begin{align}
	\mathcal{F}(\rho)
	&= \int_3^4\big(16-4m(t-3) \big) r(t) \mathrm{d}t + \int_4^T\left(\rho_m'^2(t) -m\rho_m(t) \right) r(t) \mathrm{d}t - m(4-\rho_m(4))\int_4^Tr(t)\mathrm{d}t \nonumber \\
	&\leq (16-0) \cdot 10^{-10} + \int_{4}^T\left(\rho_m'^2(t) -m\rho_m(t) \right) r(t) \mathrm{d}t - 0 \label{eq:f(rho) estimate}
	\end{align}
	so
	\begin{align}
	\mathcal{F}(\rho)-\mathcal{F}(\rho_m)
	&\leq 16 \cdot 10^{-10}
	- \int_0^4\left(\rho_m'^2(t) -m\rho_m(t) \right) r(t)\,\mathrm{d}t \nonumber\\
	&\leq 16 \cdot 10^{-10}
	- \int_1^2\rho_m'^2(t)  r(t)\,\mathrm{d}t + m \int_0^4 \rho_m(t)r(t)\,\mathrm{d}t\nonumber \\
	&\leq 16 \cdot 10^{-10}
	- \int_1^2\rho_m'^2(t)  r(t)\,\mathrm{d}t + m \int_0^4 2t\,\mathrm{d}t\nonumber \\
	&= 16 \cdot 10^{-10} -\int_1^2\rho_m'^2(t)  r(t)\,\mathrm{d}t + 16 m. \label{eq:competition}
	\end{align}
	We are left to bound from below the integral term in \eqref{eq:competition}. By the Hölder inequality we have
	\begin{equation*}
	\int_1^2 \rho_m' \leq \left(\int_1^2\frac{1}{r}\right)^{\sfrac12}\left(\int_1^2\rho_m'^2 r\right)^{\sfrac12},
	\end{equation*}
	and so
	\begin{align}
	\int_1^2\rho_m'^2 r
	&\geq \frac{(\rho_m(2)-\rho_m(1))^2}{\int_{1}^2\frac{1}{r}}  \geq (\rho_m(2)-\rho_m(1))^2, \label{eq:holder}
	\end{align}
	by the construction of $r$. Now we use the fact $u_m^d = u_m$, which means that $\abs{\nabla u_m^d} \leq 1$, and so $\abs{\rho_m'} \leq1$. With the running assumption $\rho_m(4)\geq 3.5$, this implies $\rho_m(2)\geq 1.5$. As $\rho_m(1)\leq 1$, we get $\rho_m(2)-\rho_m(1)\geq 0.5$. Then, \eqref{eq:holder} and
	\eqref{eq:competition} yield
	\begin{equation}
	\mathcal{F}(\rho)-\mathcal{F}(\rho_m) \leq 16 \cdot 10^{-10} - 0.25 + 16 m.
	\end{equation}
	Recalling that we have chosen $m = 10^{-10}$, it contradicts the minimality of $\rho_m$ in \eqref{eq:one dimensional problem}.

	\emph{Case two: $\rho_m(4)\leq 3.5$.} We use the same competitor $\rho$ as in case one. We even perform similar estimates, the only difference being that we don't estimate the term $- m(4-\rho_m(4))\int_{(4,T)}r(t)\mathrm{d}t$ by $0$ as in \eqref{eq:f(rho) estimate}. Thus \eqref{eq:competition} becomes instead:
	\begin{align}
	\mathcal{F}(\rho)-\mathcal{F}(\rho_m)
	&\leq 16 \cdot 10^{-10} -\int_1^2\rho_m'^2(t)  r(t) \mathrm{d}t + 16 m  - m(4-\rho_m(4))\int_4^Tr(t)\mathrm{d}t.\nonumber\\
	&\leq 16 \cdot 10^{-10} + 16 m - 0.5 m \int_4^Tr(t)\mathrm{d}t \nonumber\\
	& \leq 16 \cdot 10^{-10} + 16 m - 0.5 m \int_5^{T-1}r(t)\mathrm{d}t. \nonumber
	\end{align}
	Recalling that we have chosen $m = 10^{-10}$, $T=10^{10}$ and $r\geq 1$ between $5$ and $T-1$, it contradicts the minimality of $\rho_m$ in \eqref{eq:one dimensional problem}. This concludes the proof.
\end{proof}

\bibliographystyle{plain}
\bibliography{biblio}

\bigskip\noindent
François Générau:\\
Laboratoire Jean Kuntzmann (LJK),
Universit\'e Grenoble Alpes\\
Bâtiment IMAG, 700 avenue centrale,
38041 Grenoble Cedex 9 - FRANCE\\
{\tt francois.generau@univ-grenoble-alpes.fr}

\bigskip\noindent
\'Edouard Oudet:\\
Laboratoire Jean Kuntzmann (LJK),
Universit\'e Grenoble Alpes\\
Bâtiment IMAG, 700 avenue centrale,
38041 Grenoble Cedex 9 - FRANCE\\
{\tt edouard.oudet@univ-grenoble-alpes.fr}

\bigskip\noindent
Bozhidar Velichkov:\\
Dipartimento di Matematica, Universit\`a di Pisa\\
Largo Bruno Pontecorvo, 5, 56127 Pisa - ITALY\\
{\tt bozhidar.velichkov@gmail.com}

\end{document}